\documentclass{amsart}
\usepackage{mathrsfs}
\usepackage{amssymb,latexsym,amsmath,amsthm,enumitem,mathrsfs}
\usepackage{hyperref}
\usepackage{color}
\usepackage{stmaryrd}
\usepackage{mathrsfs}
\usepackage{lineno}
\usepackage{bbm}
\usepackage{scalerel,stackengine}
\usepackage{tikz}
\usepackage{enumitem}
\usepackage{ esint }% See geometry.pdf to learn the layout options. There are lots.
\usepackage{graphicx}
\usepackage{amssymb}
\newtheorem{theorem}{Theorem}
\newtheorem{lemma}[theorem]{Lemma}
\newtheorem{remark}[theorem]{Remark}
\newtheorem{proposition}[theorem]{Proposition}

\newtheorem{definition}[theorem]{Definition}

\newtheorem{corollary}[theorem]{Corollary}
\newcommand{\Thm}{\textup{Thm}}
\theoremstyle{definition}

\newcommand{\cC}{\mathcal{C}}
\newcommand{\cD}{\mathcal{D}}

\newcommand{\cH}{\mathcal{H}}

\newcommand{\cT}{\mathcal{T}}

\newcommand{\RR}{\mathbb R}

\newcommand{\T}{\mathbb T}

\renewcommand{\S}{\mathbb S}

\newcommand{\e}{\varepsilon}

\def\set4{\mathcal I}
\def\tup14{(1,2,3,4)}

\newtheorem*{comm*}{Comment}
\newtheorem*{lemma*}{Lemma}

\newcommand{\supp}{\mathrm{supp}}

\newcommand{\R}{\mathbb{R}}

\newcommand{\de}{\delta} %%%%%\delta
\newcommand{\De}{\Delta} %%%%%\delta

\newcommand{\wh}{\widehat}
\newcommand{\wt}{\widetilde}
\newcommand{\si}{\sigma}

\newcommand{\Tau}{\mathcal T}

\newcommand{\dist}{\textup{dist}}

\begin{document}

% \author{Larry Guth}
% \address{Department of Mathematics\\
% Massachusetts Institute of Technology\\
% Cambridge, MA 02142-4307, USA}
% \email{lguth@math.mit.edu}

% \author{Shengwen Gan}
% \address{Department of Mathematics\\
% Massachusetts Institute of Technology\\
% Cambridge, MA 02142-4307, USA}
% \email{shengwen@mit.edu}

% \author{Dominique Maldague}
% \address{Department of Mathematics\\
% Massachusetts Institute of Technology\\
% Cambridge, MA 02142-4307, USA}
% \email{dmal@mit.edu}

 \author{Paige Bright}
 \address{Department of Mathematics\\
 Massachusetts Institute of Technology\\
 Cambridge, MA 02142-4307, USA}
 \email{paigeb@mit.edu}
 
 \author{Shengwen Gan}
 \address{Department of Mathematics\\
 Massachusetts Institute of Technology\\
 Cambridge, MA 02142-4307, USA}
 \email{shengwen@mit.edu}


\date{}

\title[Exceptional set estimates]{Exceptional set estimates for radial projections in $\R^n$}
\maketitle

\begin{abstract}
We prove two conjectures in this paper.

\noindent
The first conjecture is by Lund, Pham and Thu:
Given a Borel set $A\subset \R^n$ such that $\dim A\in (k,k+1]$ for some $k\in\{1,\dots,n-1\}$. For $0<s<k$, we have
\[
\dim(\{y\in \R^n \setminus A\mid \dim (\pi_y(A)) < s\})\leq \max\{k+s -\dim A,0\}.
\]
The second conjecture is by Liu: Given a Borel set $A\subset \R^n$,
then 
\[
\dim (\{x\in \R^n \setminus A \mid \dim(\pi_x(A))<\dim A\}) \leq \lceil \dim A\rceil.
\]
    % For each $V\in G(m,n)$, let $\pi_V:\RR^n \to V$ be the orthogonal projection onto $V$. For a Borel set $A\subset \RR^n$, we prove that $\dim(\pi_V(A)) = \min\{m,\dim A\}$ for a.e. $V\in G(m,n)$. More generally, we prove two exceptional set estimates. For $A\subset \R^n$, and $0\le s< m$, define $E_s(A):=\{V\in G(m,n) \mid \dim(\pi_V(A))<s\}$. We have $\dim(E_s(A))\le m(n-m)+s-\dim(A)$ and $\dim(E_s(A)) \le m(n-m-1)+s$. For each $x\in \R^n$, let $\pi_x:\R^n\setminus\{x\}\to B^n(x,1)$ be the radial projection. For a Borel set $A\subset \R^n$, we make progress on Liu's conjecture.
\end{abstract}

\section{Introduction}

In this paper, we study the radial projections in $\R^n$.

Let $G(m,n)$ be the set of $m$-dimensional subspaces in $\R^n$, which is also known as the Grassmannian. For $V\in G(m,n)$, define $\pi_V:\RR^n\rightarrow V$ to be the orthogonal projection onto $V$.
Given $x\in \R^n$, define $\pi_x:\R^n\setminus\{x\}\rightarrow \mathbb S^{n-1}$ to be the radial projection centered at $x$:
\[
\pi_x(y) = \frac{y-x}{|y-x|}.
\]

We first discuss some background of the projection theory.
We use $\dim X$ to denote the Hausdorff dimension of the set $X$. 
There is a classical result proved by Marstrand \cite{marstrand1954some}, who showed that if $A$ is a Borel set in $\R^2$, then the projection of $A$ onto almost every line through the origin has Hausdorff dimension $\min\{1,\dim A\}$. This was generalized to higher dimensions by Mattila \cite{mattila1975hausdorff}, who showed that if $A$ is a Borel set in $\R^n$, then the projection of $A$ onto almost every $k$-plane through the origin has Hausdorff dimension $\min\{k,\dim A\}$. It turns out that one can obtain some finer results which are known as the exceptional set estimates. The exceptional set estimates give a bound on the set of directions where the projection is small.
There are two types of exceptional set estimates known as the Falconer-type estimate and Kaufman-type estimate.

Suppose $A\subset \RR^n$ is a Borel set of Hausdorff dimension $\alpha$. For $0\le s< \min\{m,\alpha\}$, define the exceptional set
\[
E_s(A) =\{V \in G(m,n) \mid \dim(\pi_V(A))<s \}.
\]
Then we have

\begin{enumerate}[label=(\roman*)]
    \item (Falconer-type) $\dim (E_s(A))\le \max\{m(n-m)+s-\alpha,0\}.$
    \item (Kaufman-type) $\dim (E_s(A))\le m(n-m-1)+s$.
\end{enumerate}
The original paper of Falconer and Kaufman are \cite{falconer1982hausdorff}, \cite{kaufman1968hausdorff}, where they considered the case when $n=2$. The Falconer-type estimate in higher dimensions was proved by Peres and Schlag \cite{peres2000smoothness}. We also recommend Theorem 5.10 in \cite{mattila2015fourier} for the proofs of these two types of the exceptional set estimates.

In this paper, we study the exceptional set estimates for the radial projections.
We first state our theorems.

\begin{theorem}\label{radfalconer}
Let $A\subset \R^n$ be a Borel set such that $\alpha = \dim A\in (k,k+1]$ for some $k\in\{1,\dots,n-1\}$. Fix $0<s<k$ and let 
\[
E_s(A) := \{y\in \R^n \setminus A\mid \dim (\pi_y(A)) < s\}.
\]
Then, 
\[
\dim(E_s(A))\leq \max\{k+s -\alpha,0\}.
\]
\end{theorem}

\begin{theorem}\label{liusbound}
Let $A\subset \R^n$ be a Borel set such that $\alpha = \dim A\in (k-1,k]$ for some $k\in\{1,\dots,n-1\}$. 
Define the exceptional set
\[
E(A) := \{x\in \R^n \setminus A \mid \dim(\pi_x(A))<\alpha\}.
\]
Then we have 
\[
\dim (E(A)) \leq k.
\]
\end{theorem}

Theorem \ref{liusbound} is sharp. If we let $A$ be an $\alpha$-dimensional subset of $\R^k$, we see that $E(A)=\R^k\setminus A$ which has dimension $k$.

We remark that Theorem \ref{radfalconer} is a conjecture made by Lund, Pham and Thu (see \cite[Conjecture 1.2]{lund2022radial}); Theorem \ref{liusbound} is made by Liu (see \cite[Conjecture 1.2]{liu2019hausdorff}).

Recently, Orponen and Shmerkin \cite{orponenshmerkin2022exceptional} proved the $n=2$ case for both Theorem \ref{radfalconer} and Theorem \ref{liusbound}. Their proof of Theorem \ref{radfalconer} (when $n=2$) is based on a Furstenberg-type estimate due to Fu and Ren \cite{fu2021incidence}. Then by a swapping trick, they are able to prove Theorem \ref{liusbound} (when $n=2$). In this paper, we prove the Theorems for all dimensions. We remark that the upper bound in Theorem \ref{radfalconer} is a Falconer-type bound.

Theorem \ref{radfalconer} is a result of
Proposition \ref{1prop} and Proposition \ref{2prop}. Finally, the proof of Theorem \ref{liusbound} is based on Proposition \ref{discliu} and a trick of Orponen and Shmerkin \cite{orponenshmerkin2022exceptional}.

We talk about the structure of the paper. In Section \ref{sec2}, we prove Theorem \ref{radfalconer}. In Section \ref{sec3}, we prove Theorem \ref{liusbound}.

\subsection{Some notations}
We will frequently use the following definitions.
\begin{definition}\label{dessetsd1}
For a number $\de>0$ and any set $X$ (in a metric space), we use $|X|_\de$ to denote the maximal number of $\de$-separated points in $X$.
\end{definition}

\begin{definition}\label{dessetsd2}
Let $\de,s>0$. We say $A\subset \R^n$ is a $(\de,s,C)$\textit{-set} if it is $\de$-separated and satisfies the following estimate:
\begin{equation}\label{deltas}
    \# (A\cap B_r(x)) \le C (r/\de)^s.
\end{equation}
for any $x\in \R^n$ and $1\ge r\geq \de$. In this paper, the constant $C$ is not important, so we will just say $A$ is a $(\de,s)$-set if
\[ \#(A\cap B_r(x))\lesssim (r/\de)^s \]
for any $x\in \R^n$ and $1\ge r\geq \de$.
\end{definition}

\begin{remark}
    {\rm
    Since the condition \eqref{deltas} is for scales $\ge \de$, we can abuse the notation to define: for $A'=\sqcup B_\de$ being a union of disjoint $\de$-balls, we say $A$ is a $(\de,s)$-set if
    \[ \#\{B_\de: B_\de\subset A\cap B_r(x)\}\lesssim (r/\de)^s. \]
    This definition is consistent with the previous definition: If $A$ is a $(\de,s)$-set, then the $\de$-neighborhood of $A$ is also a $(\de,s)$-set in the new sense; conversely, if $A'$ is a disjoint union of $\de$-balls and is a $(\de,s)$-set in the new sense, then the set of centers of the $\de$-balls in $A'$ is a $(\de,s)$-set in the old sense. Therefore, it makes sense to say a set $A$ is a $(\de,s)$-set if $A$ is $\de$-separated or $A$ is a disjoint union of $\de$-balls.}
\end{remark}

\begin{remark}
{\rm
Throughout the rest of this paper, We will use $\#E$ to denote the cardinality of a set $E$ and $\lvert \cdot\rvert$ to denote the measure of a region.}
\end{remark}

We state two lemmas:
\begin{lemma}\label{frostmans}
Let $\de,s>0$ and let $B\subset \R^n$ be any set with $\mathcal H^s_\infty(B) =: \kappa >0$. Then, there exists a $(\de,s)$-set $P\subset B$ with $\#P\gtrsim \kappa \de^{-s}$.
\end{lemma}
\begin{proof}
See \cite{fassler2014restricted} Lemma 3.13.
\end{proof}

\begin{lemma}\label{frostmans3} Fix $a>0$.
Let $\nu$ be a probability measure satisfying $\nu(B_r)\lesssim r^a$ for any $B_r$ being a ball of radius $r$. If $A$ is a set satisfying $\nu(A)\ge \kappa$ ($\kappa>0$), then for any $\de>0$ there exists a subset $F\subset A$ such that $F$ is a $(\de,a)$-set and $\# F\gtrsim \kappa \de^{-a}$.
\end{lemma}
\begin{proof}
By the previous lemma, we just need to show $\cH^a_\infty(A)\gtrsim \kappa$. We just check it by definition. For any covering $\{B\}$ of $A$, we have
$$\kappa\le \sum_B\nu(B)\lesssim \sum_B r(B)^a. $$
Ranging over all the covering of $A$ and taking infimum, we get $$\kappa\lesssim \cH^a_\infty(A).$$
\end{proof}

\subsection{\texorpdfstring{$\de$}{Lg}-tube and \texorpdfstring{$\de$}{Lg}-slab}\hfill

One of the main geometric objects we will study is the so-called $\de$-tube. In $\R^n$, we call $T$ a $\de$-tube, if $T$ is a tube of radius $\de$ and length $1$. If $T'$ is a convex set that is comparable to a $\de$-tube $T$ (here when we say $T$ and $T'$ are comparable, it means $10^{-1}T\subset T'\subset 10T$), then we also call $T'$ a $\de$-tube. Therefore, if $T$ is a rectangle of dimensions $\sim \de\times\dots\times \de\times 1$, then $T$ is also a $\de$-tube.

For two $\de$-tubes $T$ and $T'$, we say they are comparable, if $10^{-1}T\subset T'\subset 10T$. We say they are essentially distinct, if they are not comparable.

In this paper, we will frequently encounter the following situation. There are two finite sets $E,F\subset B^n(0,1)$ (here $B^n(0,1)$ is the unit ball in $\R^n$ centered at the origin).  Each of $E$ and $F$ is contained in a ball of radius $1/8$, and $\dist(E,F)>1/2$. We use letter $y$ to denote the points in $E$, $x$ to denote the points in $F$. $F$ is a $(\de,\alpha)$-set, and $E$ is a $(\de,t)$-set. We can view $F$ (or $E$) as a $\de$-discretized version of $A$ (or $E_s(A)$) in Theorem \ref{radfalconer}. If $y\in E$ is in the exceptional set, then the maximal 
$\de$-separated subset of $\pi_y(F)$ is roughly a $(\de,s)$-set in $\S^{n-1}$. We would like to use another geometric object to characterize $\pi_y(F)$. For every $\omega\in \S^{n-1}$, we can define a tube $T_\omega$ which is the $\de$-neighborhood of the line segment $\{y+t\omega: t\in[0,1]\}$. Roughly speaking, $T_\omega$ is a tube of dimensions $\sim \de\times \dots\times \de\times 1$ pointing to the direction $\omega$ and passing $y$.
In this correspondence, a maximal 
$\de$-separated subset of $\pi_y(F)$ gives rise a set of $\de$-tubes $\T^y$ that pass through $y$, and $\bigcup_{T\in \T^y}T \supset F $. We call $\T^y$ a \textit{bush} centered at $y$. And the $(\de,s)$ condition of $\pi_y(F)$ transfers to $\T^y$ which says that: If $T_r$ is a $r\times \dots\times r\times 1$-tube passing through $y$, then there are $\lesssim (r/\de)^s$ many tubes in $\T^y$ that are contained in $T_r$ ($\de\le r\le 1$). When we call a bush $\T^y$ centered at $y$ a $(\de,s)$-set, we mean that $\pi_y(\bigcup_{T\in\T^y}T)\subset \S^{n-1}$ is a $(\de,s)$-set. 

We have discussed the notion of a bush centered at $y$ and the definition for a bush to be a $(\de,s)$-set. We also need to consider another type of bush called the \textit{truncated bush}. If $\T^y$ is a bush centered at $y$, then for each $T\in \T^y$ we define the truncated tube
\[ \wt T=T\setminus B_{1/2}(y). \]
By truncation, $\T^y$ gives rise a truncated bush $\wt\T^y$ centered at $y$. The reason we do this truncation is that the tubes in $\wt \T^y$ are now essentially disjoint. This will be helpful in estimating the upper bound of integrals like
$\int_{\R^n} (\sum_y \sum_{T\in \wt\T^y} 1_T)^2$.

We will also study the geometric object called the $k$-dimensional $\de$-slab. They are of dimensions $\underbrace{\de\times \dots\times\de}_{n-k\textup{~times}}\times \underbrace{1\times \dots\times 1}_{k \textup{~times}}$. They are morally the $\de$-neighborhood of a $k$-dimensional plane truncated in a ball of radius $1$. In particular, a $\de$-tube is a $1$-dimensional $\de$-slab. 

\bigskip

\begin{sloppypar}
\noindent {\bf Acknowledgement.}
The research was done during the MIT SPUR program.
We would like to thank the MIT SPUR program. We would also like to thank Prof. Larry Guth for suggesting the problem
and helpful discussions, and Prof. Ankur Moitra and Prof. David Jerison for helpful discussions. We would also like to thank Bochen Liu for pointing out a gap in the previous version of the proof.
\end{sloppypar}

\section{Falconer-type estimates for radial projections}\label{sec2}

In this section of the paper, we prove Theorem \ref{radfalconer}.
%In this section of the paper, we will make progress on an open conjecture by Bochen Liu\cite{liu2019hausdorff} regarding the radial bound. This part of the paper is currently very unpolished as I work through the details, which I hope you'll forgive for the time being. I hope the propositions lead way to proving the actual conjecture.

We introduce some notations.
Fix $0\le\si$, $\de>0$. For a bounded set $E\subset \R^n$,
define 
$$ \cH_{\de,\infty}^s(E):=\inf\left\{ \sum_j r(D_j)^s: E\subset \cup_j D_j \right\}, $$
where the infimum runs over the coverings of $E$ by dyadic cubes $\{D_j\}$ with length $\ge \de$, and $r(D)$ denotes the length of the cube.
One may compare with the definition of
$$ \cH_{\infty}^s(E):=\inf\left\{ \sum_j r(D_j)^s: E\subset \cup_j D_j \right\}, $$
where the infimum runs over the coverings of $E$ by dyadic cubes $\{D_j\}$ (without assuming length $\ge\de$).

We state three useful lemmas about $\cH_{\de,\infty}^s$.
For any dyadic number $\de\le 1$, let $\cD_\de$ be the lattice $\de$-cubes in $[0,1]^m$.

\begin{lemma}\label{usefullemma}
Suppose $X\subset [0,1]^m$ with $\dim X< s$. Then for any $\e>0$, there exist dyadic cubes $\cC_{2^{-k}}\subset \cD_{2^{-k}}$ $(k>0)$ so that 
\begin{enumerate}
    \item $X\subset \bigcup_{k>0} \bigcup_{D\in\cC_{2^{-k}}}D, $
    \item $\sum_{k>0}\sum_{D\in\cC_{2^{-k}}}r(D)^s\le \e$,
    \item $\cC_{2^{-k}}$ satisfies the $s$-dimensional condition: For $l<k$ and any $D\in \cD_{2^{-l}}$, we have $\#\{D'\in\cC_{2^{-k}}: D'\subset D\}\le 2^{(k-l)s}$.
\end{enumerate}
\end{lemma}

\begin{proof}
See \cite{GGM} Lemma 2.
\end{proof}

\begin{remark}
{\rm 
Besides $[0,1]^m$, this Lemma also works for other compact metric spaces, for example $\S^n$ and $G(m,n)$, which we will use throughout the rest of the paper.
}
\end{remark}

\begin{lemma}\label{usefullemma2}
Suppose $X\subset [0,1]^m$. Then there exist dyadic cubes $$\cC=\bigsqcup_{k=0}^{\log_2\de^{-1}} \cC_{2^{-k}}$$ (with $\cC_{2^{-k}}\subset \cD_{2^{-k}}$) that cover $X$ and
\begin{enumerate}
    \item $\sum_{D\in\cC}r(D)^s= \cH^s_{\de,\infty} (X)$,
    \item $\cC_{2^{-k}}$ satisfies the $s$-dimensional condition: For $l<k$ and any $D\in \cD_{2^{-l}}$, we have $\#\{D'\in\cC_{2^{-k}}: D'\subset D\}\le 2^{(k-l)s}$. In particular, $\cH^s_{2^{-k},\infty}(\cup_{D\in\cC_{2^{-k}}} D)=\#\cC_{2^{-k}} 2^{-ks}$.
\end{enumerate}
\end{lemma}
\begin{proof}
This lemma looks like Lemma \ref{usefullemma}, but it is much easier since we only care about the scales $\ge \de$. We just choose $\cC$ to be the covering that attain the ``inf" in the definition of $\cH^s_{\de,\infty}(X)$. It is not hard to check the two properties are satisfied.
\end{proof}

The next lemma is \cite{fassler2014restricted} \textup{Proposition A.1}. Though it is stated for $\cH^s_\infty$ there, the proof also works for $\cH^s_{\de,\infty}$.
\begin{lemma}\label{usefullemma2.5}
Suppose $X\subset[0,1]^m$, with $\cH^s_{\de,\infty}(X)=\kappa>0$. Then there exists a $(\de,s)$-subset of $X$ with cardinality $\gtrsim \kappa \de^{-s}$.
\end{lemma}

We also have the following lemma saying that the lemma above can be reversed.
\begin{lemma}\label{usefullemma3}
Suppose $X\subset [0,1]^m$ is a $(\de,s)$-set with $\#X\ge \kappa \de^{-s}$. Then, $\cH^s_{\de,\infty}(X)\gtrsim \kappa$. In particular, by Lemma \ref{usefullemma2.5}, this implies that for any $\de\le \De\le 1$, $X$ contains a subset $X'$ which is a $(\De,s)$-set and satisfies $\#X'\gtrsim \kappa \De^{-s}$;
and also implies that for any $u\le s$, $X$ contains a subset $X'$ which is a $(\De,u)$-set and satisfies $\#X'\gtrsim \kappa \De^{-u}$.
\end{lemma}
\begin{proof}
Assuming our $(\de,s)$-set $X$ satisfies $\# X\ge\kappa\de^{-s}$, we are going to show $\cH^s_{\de,\infty}(X)\gtrsim \kappa$.
Let $\cC$ be the covering of $X$ that attains ``inf" in the definition of $\cH^s_{\de,\infty}(X)$. Also let $\cC_\De\subset \cC$ be the set of $\De$-cubes. We write $X=\bigsqcup_\De X_\De$, where $X_\De$ is the points in $X$ covered $\cC_\De$. By the definition of $(\de,s)$-set, each $\De$-cube contains $\lesssim (\frac{\De}{\de})^s$ many points from $X_\De$. We have $\#\cC_\De\gtrsim(\frac{\de}{\De})^s \#X_\De$. We see that
$$ \cH^s_{\de,\infty}(X)=\sum_{\De\ge\de} \De^{s}\#\cC_\De\gtrsim \de^s\#X=\kappa. $$
\end{proof}

\begin{remark}
    {\rm
    We see that when $X$ is a $(\de,s)$-set, then $\#X\gtrsim \de^{-s+\e}$ and $\cH^s_{\de,\infty}(X)\gtrsim \de^\e$ are equivalent.
    }
\end{remark}

We recall Theorem \ref{radfalconer} here.
\begin{theorem}\label{thmradial}
Let $A\subset \R^n$ be a Borel set such that $\alpha = \dim A\in (k,k+1]$ for some $k\in\{1,\dots,n-1\}$. Fix $0<s<k$ and let 
\[
E_s(A) := \{y\in \R^n \setminus A\mid \dim (\pi_y(A)) < s\}.
\]
Then, 
\[
\dim(E_s(A))\leq \max\{k+s -\alpha,0\}.
\]
\end{theorem}

We will actually prove the following $\de$-discretized version which is a generalization of \cite[Proposition 4.2]{orponenshmerkin2022exceptional}.

\begin{theorem}\label{discradial}
Let $0<\si <k$, $a\in (k,k+1]$ for some $k\in\{1,\dots,n-1\}$ and $t>\max\{k+\si - a, 0\}$. Let $\eta\in (0,1/10)$. Then for $\e$ and $\de$ small enough depending on $\eta,\si,a$, and $t$, we have the following result.  

Let $E,F\subset B^n(0,1)$ be a $(\de,t)$-set and a $(\de,a)$-set respectively, with $\#E \gtrsim \de^{-t+\e}$, $\#F \gtrsim \de^{-a+\e}$. We also assume: each of $E$ and $F$ lies in a ball of radius $1/1000$ and $\dist(E,F)\ge 3/4$.
Then, there exists $y\in E$ such that for all $F'\subset F$ with $\mathcal \# F' \geq \de^{\e}\#F$, we have
\[
\cH^\si_{\de,\infty}(\pi_y(F'))>\de^\eta.
\]
\end{theorem}

We first show that Theorem \ref{discradial} implies Theorem \ref{thmradial}.

\begin{proof}[Proof that Theorem \ref{discradial} implies Theorem \ref{thmradial}]\label{pf}
We first do a reduction to localize $A$. For $\alpha_1<\alpha$, we say $x\in A$ is an $\alpha_1$-dense point of $A$ if $\dim(A\cap B_{r}(x))\ge \alpha_1$ for any $r>0$.
We notice a fact: for $\alpha_1<\alpha$, $A$ has infinitly many $\alpha_1$-dense points; otherwise, $A$ can be covered by a finite set and countable union of sets with dimension less than $\alpha_1$, which contradicts $\dim A=\alpha$.

Fix $\alpha_1<\alpha$ that is sufficiently close to $\alpha$ (we will later let $\alpha_1\rightarrow \alpha$). We can find $\alpha_1$-dense points $x_1,x_2$ of $A$. Since our problem is scaling-invariant, we can assume $|x_1-x_2|=99/100$. We let $A_1=A\cap B_{1/1000}(x_1)$, $A_2=A\cap B_{1/1000}(x_2)$, and then $\dim(A_1),\dim(A_2)\ge \alpha_1$. We only need to show for any ball $B_{1/1000}$ of radius $1/1000$, $E_s(A)\cap B_{1/1000}$ has dimension $\le\max\{k+s-a,0\}$. Since $\dist(A_1,A_2)>98/100$, we have either $\dist(B_{1/1000},A_1)>3/4$ or $\dist(B_{1/1000},A_2)>3/4$. We may assume $\dist(B_{1/1000},A_1)>3/4$. It suffices to show that the set
$$ E':=E_s(A_1)\cap B_{1/1000}=\{y\in B_{1/1000}:\dim(\pi_y(A_1))<s\} $$
has dimension $\le \max\{k+s-\dim(A_1),0\}$.
From the reduction, these sets satisfy certain separation properties: 
\begin{align}
   \textup{each ~of~} A_1 \textup{~and~} E' \textup{~lies~in~some~ball~of~radius~} 1/1000,\\  A_1,E'\subset B^n(0,1),\ \  \dist(A_1,E')\ge 3/4.
\end{align}
 (We remark that the numerology about the radii of balls or the distance between sets are not important. For example, we only need $A_1, E'$ to be contained in a ball of bounded radius and the distance between $A_1$ and $E'$ are bigger than some nonzero constant.)

We choose $t<\dim(E'),a<\dim(A_1)$. Then $\cH^t_{\infty}(E')>0$, and by Frostman's lemma there exists a probability measure $\nu_{A_1}$ supported on $A_1$ satisfying $\nu_{A_1}(B_r)\lesssim r^a$ for any $B_r$ being a ball of radius $r$. We only need to prove $t\le \max\{k+s-\alpha_1,0\}$, since then we can send $a\rightarrow \dim(A_1),t\rightarrow \dim(E')$. For the sake of contradiction, assume that $t> \max\{k+s-a,0\}$. Thus, we can find $\si>s$ so that $t> \max\{k+\si-a,0\}$. Set $\eta=\si-s>0$.
Now we fix $a,t$, so we may assume $\cH^t_{\infty}(E')\sim 1$ is a constant. 

% For any $y\in E'$, we have $\dim(\pi_y(A_1))<s$. Since $0=\cH^s_\infty(\pi_y(A_1))=\lim_{\de\rightarrow 0}\cH^s_{\de,\infty}(\pi_y(A_1))$, we find a subset of $E'$ such that we have $\cH^t_\infty(E')\sim 1$ and for small enough $\de$: $\cH^s_{\de,\infty}(\pi_x(A_1))\le 1$ for all $y\in E'.$
Fix a $y\in E'$.
applying Lemma \ref{usefullemma} to $\pi_y(A_1)$, we obtain a set of dyadic caps $\cC_y=\bigsqcup_j \cC_{y,j}$ in $\S^{n-1}$ that cover $\pi_y(A_1)$. Here each $\cC_{y,j}$ is a set of $2^{-j}$-caps that satisfy the $s$-dimensional condition (see Lemma \ref{usefullemma} (3)) as $\dim (\pi_y(A_1))<s$. Also, the radius of these caps is less than $\e_\circ$ which is any given small number. 

By the $s$-dimensional condition of $\cC_{y,j}$, we have
$$ \cH^s_{2^{-j},\infty}\left(\bigcup_{C\in\cC_{y,j}}C\right)=\# \cC_{y,j} 2^{-js}\le 1. $$
Therefore, we have
\begin{equation}\label{radcontr}
    \cH^\si_{2^{-j},\infty}\left(\bigcup_{C\in\cC_{y,j}}C\right)\le \#\cC_{y,j}2^{-j\si}\le 2^{-j\eta}.
\end{equation}

\begin{figure}
\begin{tikzpicture}[x=0.75pt,y=0.75pt,yscale=-1,xscale=1]
%uncomment if require: \path (0,271); %set diagram left start at 0, and has height of 271

%Shape: Ellipse [id:dp015175412233626062] 
\draw   (242,73) .. controls (242,52.01) and (277.82,35) .. (322,35) .. controls (366.18,35) and (402,52.01) .. (402,73) .. controls (402,93.99) and (366.18,111) .. (322,111) .. controls (277.82,111) and (242,93.99) .. (242,73) -- cycle ;
%Shape: Ellipse [id:dp9041880446429325] 
\draw   (243,192) .. controls (243,171.01) and (278.82,154) .. (323,154) .. controls (367.18,154) and (403,171.01) .. (403,192) .. controls (403,212.99) and (367.18,230) .. (323,230) .. controls (278.82,230) and (243,212.99) .. (243,192) -- cycle ;
%Shape: Circle [id:dp5596862970264425] 
\draw  [dash pattern={on 4.5pt off 4.5pt}] (255,73) .. controls (255,59.19) and (266.19,48) .. (280,48) .. controls (293.81,48) and (305,59.19) .. (305,73) .. controls (305,86.81) and (293.81,98) .. (280,98) .. controls (266.19,98) and (255,86.81) .. (255,73) -- cycle ;
%Shape: Rectangle [id:dp5436700681431033] 
\draw   (274,98.33) -- (286,98.33) -- (286,240) -- (274,240) -- cycle ;
%Shape: Rectangle [id:dp5474205934858001] 
\draw   (342.08,227.02) -- (330.8,231.46) -- (286,98.33) -- (297.28,93.9) -- cycle ;
%Shape: Rectangle [id:dp058313028278387025] 
\draw   (280,98.33) -- (291.82,96.08) -- (314.57,235.59) -- (302.75,237.85) -- cycle ;
%Flowchart: Summing Junction [id:dp36616720138444125]

\foreach \Point in {(278,73), (280,213.5), (305,215.5), (280,188.5), (300,189.5), (327.5,204)}{
    \node at \Point {\textbullet};
}

% Text Node
\draw (408,47.4) node [anchor=north west][inner sep=0.75pt]    {$E'$};
% Text Node
\draw (409,200.4) node [anchor=north west][inner sep=0.75pt]    {$A_1$};
% Text Node
\draw (282,67.4) node [anchor=north west][inner sep=0.75pt]    {$y$};

\end{tikzpicture}
\caption{$\T_{y,j}$ in the radial projection}
\label{radialproj}
\end{figure}
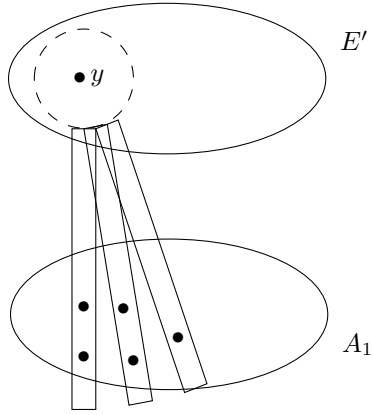

For each cap $C\in\cC_y$, consider $\pi_y^{-1}(C)\cap \{x\in\R^n:1-\frac{1}{100}\le|x-y|\le 1\}$ which is a tube. We obtain a collection of finitely overlapping tubes
$$ \T_y=\bigsqcup_j \T_{y,j} $$
that cover $A_1$ (see Figure \ref{radialproj}). This is a truncated bush centered at $y$.
Here, each tube has its coreline passing through $y$ and at distance $\sim 1$ from $y$. The tubes in $\T_{y,j}$ have dimensions $\sim 2^{-j}\times\dots\times2^{-j}\times 1$.

For this fixed $y\in E'$, there exists a $j(y)\ge |\log_2\e_\circ|$ such that 
\begin{equation}\label{radpigeon1}
    \nu_{A_1}\left(A_1\cap \bigcup_{T\in\T_{y,j(y)}}T\right)\ge \frac{1}{10j(y)^2}\nu_{A_1}(A_1)=\frac{1}{10j(y)^2}.
\end{equation}
We have a partition $E'=\bigsqcup_j E'_j$ where
$E'_j=\{y\in E': j(y)=j\}.$ We choose $j$ such that $\cH^t_\infty(E'_j)\gtrsim \frac{1}{j^2}$.
We let $\de=2^{-j}$. Note that $\de \leq \e_\circ$ by assumption. By Lemma \ref{frostmans}, there exists a subset $E''\subset E'_j$ which is a $(\de,t)$-set and $\# E''\gtrsim |\log\de|^{-2}\de^{-t}$. We use $\mu$ to denote the counting measure on $E''$.

Next, we consider the set $S=\{(y,x)\in E''\times A_1: x\in\bigcup_{T\in\T_{y,j}}T\}$. We also denote the $y$-section and $x$-section of $S$ by $S^y$ and $S_x$. (In Figure \ref{radialproj}, $E'$ is drawn above $A_1$, so we let $y$ be the superscript in $S^y$). By \eqref{radpigeon1}, we have $\nu_{A_1}(S^y)\ge \frac{1}{10j(y)^2}$, so we have
\begin{equation}\label{radpigeon3}
    (\mu\times\nu_{A_1})(S)\ge \frac{1}{10j^2}\mu(E'').
\end{equation}
This implies 
\begin{equation}\label{radpigeon4}
    (\mu\times\nu_{A_1})\bigg(\Big\{(y,x)\in S: \mu (S_x)\ge \frac{1}{20j^2}\mu(E'')\Big\} \bigg)\ge \frac{1}{20j^2}\mu(E'').
\end{equation}
Therefore, we have
\begin{equation}\label{radpigeon5}
    \nu_{A_1}\bigg(\Big\{x\in A_1: \mu (S_x)\ge \frac{1}{20j^2}\mu(E'')\Big\} \bigg)\ge \frac{1}{20j^2}\sim |\log\de|^{-2}.
\end{equation}
By Lemma \ref{frostmans3}, we can find a subset $F$ of $\Big\{x\in A_1: \mu (S_x)\ge \frac{1}{20j^2}\mu(E'')\Big\}$, so that $F$ is a $(\de,a)$-set and $\#F\gtrsim |\log\de|^{-2}\de^{-a}$.

Hence,
\begin{equation}
   |\log\de|^{-2}\#F\#E'' \lesssim \sum_{x\in F} \#\left\{ y\in E'': x\in \bigcup_{T\in\T_{y,j}} T \right\}=\sum_{y\in E''}\#\left\{x\in F: x\in \bigcup_{T\in\T_{y,j}} T \right\}.
\end{equation}
By pigeonholing, there exists a subset $E\subset E''$ with $\#E\gtrsim |\log\de|^{-2}\#E''\gtrsim \de^{\e/2} \de^{-t}$, so that for any $y\in E$:
$$ \#\{x\in F: x\in \bigcup_{T\in\T_{y,j}} T \}\gtrsim \de^{\e/2}\# F\ge \de^\e\#F, $$
when $\de$ is small enough.

We set $F_y:=\{x\in F: x\in \bigcup_{T\in\T_{y,j}} T \}$.
Now we use Theorem \ref{discradial} to derive a contradiction. Since $E$ is a $(\de,t)$-set with $\#E\gtrsim \de^{\e}\de^{-t}$ and $F$ is a $(\de,a)$-set with $\#F\gtrsim \de^{-a+\e}$,
Theorem \ref{discradial} yields the existence of an $y\in E$ such that $\cH^\si_{\de,\infty}(\pi_y(F_y))>\de^\eta$. This contradicts \eqref{radcontr}.
\end{proof}

Before proving Theorem \ref{discradial}, we prove two propositions. Then we show Theorem \ref{discradial} is a result of them. The first proposition is a quantitative version of Marstrand's projection theorem. The second proposition is a special case of Theorem \ref{discradial} when $k=n-1$.

\begin{proposition}\label{1prop}
Set $d_{m,n}=m(n-m)=\dim (G(m,n))$. Let $0<a<m$. Let $\eta\in (0,1/10)$. Then for $\e$ and $\de$ small enough depending on $\eta,a$, we have the following result.  

Let $F\subset B^n(0,1)$ be a $(\de,a)$-set with $\#F\gtrsim \de^\e\de^{-a}$.
Let $G\subset G(m,n)$ be a $(\de,d_{m,n})$-set, with $\#G\gtrsim \de^\e\de^{-d_{m,n}}$.
Then, there exists $V\in G$ such that for all $F'\subset F$ with $\mathcal \# F' \geq \de^{\e}\#F$, we have
\begin{equation*}
    \mathcal H^a_{\de,\infty}(\pi_V(F')) > \de^{\eta}.
\end{equation*}

Here, $\pi_V$ is the orthogonal projection onto $V$. Actually, by iteration, there exists $G_1\subset G$ with $G_1\ge 1/2 \#G$ such that any $V\in G_1$ satisfies the property above. The idea is to construct $G_1=\{V_1,\dots,V_N\}$ inductively and replace $G$ by $G\setminus G_1$, and check whether $\#(G\setminus G_1)\gtrsim \de^\e \de^{-d_{m,n}}$, and then repeat again.
\end{proposition} 

\begin{proof}
Suppose the result is not true. By contradiction, for any $V\in G$, there exists $F_V\subset F$ with $\#F_V\ge \de^\e\#F$ and
\begin{equation}\label{spcing}
    \cH^a_{\de,\infty}(\pi_V(F_V))\le \de^\eta.
\end{equation} 
By the definition of $\cH^a_{\de,\infty}$, we can find a covering of $\pi_V(F_V)$ by dyadic cubes $\{D\}$ so that $\cH^a_{\de,\infty}(\pi_V(F_V))=\sum_D r(D)^a$.
Consider $\{\pi_V^{-1}(D)\cap B^n(0,1)\}$ which are the preimages of these $\{D\}$ under $\pi_V$ truncated in the unit ball.
They actually form a covering of $F_V$:
$$ F_V\subset \bigsqcup_{\de\le \Delta\le 1}\bigcup_{ T\in\T_{V,\Delta}}T. $$
Here, each $\T_{V,\Delta}$ consists of planks of dimensions
$\underbrace{\De \times \De \times \dots \times \De}_{m \text{ times}} \times \underbrace{1\times 1\times \dots \times 1}_{n-m \text{ times}}
$ that are orthogonal to $V$. By Lemma \ref{usefullemma2}, $\T_{V,\De}$ satisfies the $a$-dimensional spacing condition (inherited from $\{D\})$: For $\De\le r\le 1$, if $T_{r}$ is a plank of dimensions $\underbrace{r \times r \times \dots \times r}_{m \text{ times}} \times \underbrace{1\times 1\times \dots \times 1}_{n-m \text{ times}}
$ that is orthogonal to $V$, then $T_{r}$ contains $\lesssim (r/\De)^a$ many planks from $\T_{V,\De}$.
Also by \eqref{spcing},
\begin{equation}\label{radnote}
     \#\T_{V,\De}\lesssim  \de^\eta\De^{-a}.
\end{equation}
We see that $\T_{V,\De}$ is non-empty only for $\De\le \de^{\eta/a}$.

Next, we will apply a standard pigeonhole argument to find a scale $\De$. Note that
\[F_V\subset \bigsqcup_{\de\le\De\le \de^{\eta/a}} \bigcup_{T\in\T_{V,\De}}T.\]
For each $V\in G$, we can find a dyadic $\De(V)\in [\de,\de^{\eta/a}]$ so that
\begin{equation}\label{pig0}
    \#(F_V\cap \bigcup_{T\in\T_{V,\De(V)}}T)\gtrsim |\log\de|^{-1} \#F_V\gtrsim \de^{\e}\#F. 
\end{equation} 
Define $G_\De=\{V\in G: \De(V)=\De\}$. We see that
\[ G=\bigsqcup_{\de\le \De\le \de^{\eta/a}}G_\De. \]
By pigeonholing again, we can find a scale $\De$, such that 
\begin{equation}\label{pig2}
    \#G_\De\gtrsim \de^\e\#G.
\end{equation} 
We fix this $\De$. Noting that $G$ is a $(\de,d_{m,n})$-set with $\#G\gtrsim \de^{\e}\de^{-d_{m,n}}$, we have that $G_\De$ is also a $(\de,d_{m,n})$-set with $\#G_\De \gtrsim \de^{2\e}\de^{-d_{m,n}}$. By Lemma \ref{usefullemma3}, we can find a subset $G'$ of $G_\De$ so that $G'$ is a $(\De,d_{m,n})$-set with $\#G'\gtrsim \de^{2\e}\De^{-d_{m,n}}$. From \eqref{pig0}, we have for any $V\in G'$ that
\begin{equation}\label{pig1}
    \#(F\cap \bigcup_{T\in \T_{V,\De}}T)\gtrsim \de^{\e}\#F\gtrsim \de^{2\e-a} .
\end{equation}

Next, we consider the set
\[ S:=\{ (x,V)\in F\times G': x\in \bigcup_{T\in \T_{V,\De}}T \}. \]
Define the sections of $S$:
\[ S_x:=\{V\in G': (x,V)\in S\},\ \ \ S_V:=\{x\in F: (x,V)\in S\}. \]
By \eqref{pig1}, we have $\#S_V\gtrsim \de^{\e}\#F$ for $V\in G'$. Then we have
\begin{equation}\label{pig3}
    \#S=\sum_{V\in G'} \#S_V\ge C^{-1} \de^{\e} \#G'\#F.
\end{equation}
Since
\[ \#\{(x,V)\in S: \#S_x\le  (2C)^{-1}\de^{2\e}\#G' \}\le (2C)^{-1} \de^{2\e} \#G'\#F\le \frac{1}{2} \#S,  \]
we have
\[ \#\{(x,V)\in S: \#S_x\ge  (2C)^{-1}\de^{2\e}\#G' \}\gtrsim \de^{\e} \#G'\#F. \]
The inequality above implies
\[ \#\{ x\in F: \#S_x \ge  (2C)^{-1}\de^{2\e}\#G' \}\gtrsim \de^{\e} \#F. \]
We define
\begin{equation}\label{deffdelta}
     F_\De:= \{ x\in F: \#S_x \ge  (2C)^{-1}\de^{2\e}\#G' \}.
\end{equation}
Noting that $F_\De$ is a $(\de,a)$-set with $\#F_\De\gtrsim \de^{2\e-a}$, by Lemma \ref{usefullemma3}, we can find a subset $F'\subset F_\De$ such that $F'$ is a $(\De,a)$-set with
\begin{equation}
    \#F'\gtrsim \de^{2\e}\De^{-a}.
\end{equation}

Let us summarize what we obtained.
We find a scale $\De\in[\de,\de^{\eta/a}]$, a $(\De,d_{m,n})$-set $G'\subset G$ with $\#G'\gtrsim \de^{2\e}\De^{-d_{m,n}}$, and a $(\De,a)$-set $F'\subset F$ with $\#F'\gtrsim \de^{2\e}\De^{-a}$, so that
\begin{enumerate}[label=(\roman*)]
    \item for each $V\in G'$, we have a set of tubes $\T_{V,\De}$ that satisfy the $a$-dimensional spacing condition and $\#\T_{V,\De}\lesssim \de^\eta\De^{-a}$ (see paragraph before \eqref{radnote}),
    \item each $x\in F'$ is contained in $\gtrsim \de^{2\e}\#G'\gtrsim \de^{4\e}\De^{-d_{m,n}}$ planks from $\bigcup_{V\in G'}\T_{V,\De}$ (see \eqref{deffdelta}).
\end{enumerate}
In the rest of the proof, we fix $\De$ and simply write $\T_{V,\De}$ as $\T_V$.

For each $V\in G'$, let $D_V$ be a
\[
\underbrace{\De^{-1} \times \De^{-1} \times \dots \times \De^{-1}}_{m \text{ times}} \times \underbrace{1\times 1\times \dots \times 1}_{n-m \text{ times}}
\]
slab centered at the origin such that the $1\times 1\times \dots \times 1$-side is orthogonal to $V$. Then, $D_V$ is the dual rectangle of the slabs in  $\T_V$.

For all $T \in \T_V$, choose a smooth bump function $\psi_{T}$ adapted to $T$ such that $\psi_{T}\geq 1$ on $T$, $\psi_{T}$ decays rapidly outside of $T$, and $\supp ~\wh \psi_{T}\subset D_V$.

Define 
\[
f_V = \sum_{T\in \T_V} \psi_{T} \hspace{.25cm}\text{and}\hspace{.25cm} f = \sum_{V\in G'} f_V.
\]
Then by the condition (ii) above, for $x\in N_{\De}(F')$, we have: 
\[f(x) \gtrsim \de^{4\e}\De^{-d_{m,n}}.\]
So, 
\begin{equation}\label{tdimlower}
     \int_{N_\De(F')} |f|^2\gtrsim  \de^{O(\e)}\De^n\De^{-a-2d_{m,n}}.
\end{equation}

We are going to find an upper bound of $\int_{N_{\De}(F')} |f|^2$ using the high-low method. Let $K$ be a large number to be determined later (we will actually choose $K\sim \de^{-O(\e)}$). Let $\eta_{\textup{low}}(\xi)$ be a smooth bump function on $B^n(0,(K\De)^{-1})$ and $\eta_{\textup{high}}(\xi) = 1-\eta_{\textup{low}}(\xi)$. We have the following high-low decomposition for $f$:
\[
f=f_{\textup{low}}+f_{\textup{high}},
\]
where $\wh f_{\textup{low}}=\eta_{\textup{low}}\wh f $ and $\wh f_{\textup{high}}=\eta_{\textup{high}}\wh f$. See Figure \ref{dualtubes} for a diagram of the high part and low part and the dual slabs.

\begin{figure}
\begin{tikzpicture}
\draw[black] (1.5,0,0) -- ++(-3,0,0) -- ++(0,-.5,0) -- ++(3,0,0) -- cycle;
\draw[black] (1.5,0,0) -- ++(0,0,-3) -- ++(0,-.5,0) -- ++(0,0,3) -- cycle;
\draw[black] (1.5,0,0) -- ++(-3,0,0) -- ++(0,0,-3) -- ++(3,0,0) -- cycle;

\draw[black] (.25,1.25,0) --
++(-.5,0,0) -- ++(0,-3,0) -- ++ (.5,0,0) --  cycle;
\draw[black] (.25,1.25,0) --
++(0,0,-3) -- ++(0,-3,0) -- ++(0,0,3) -- cycle;
\draw[black] (.25,1.25,0) --
++(-.5,0,0) -- ++(0,0,-3) -- ++(.5,0,0) -- cycle;

\foreach \Point in {(0,-.25,-1.5)}{
    \node[blue] at \Point {\textbullet};
}

\shade[ball color = blue!40, opacity = 0.3] (0,-.25,-1.5) circle (1 cm);
  \draw[blue] (0,-.25,-1.5) circle (1 cm);
\end{tikzpicture}
\caption{Dual Slabs}
\label{dualtubes}
\end{figure}
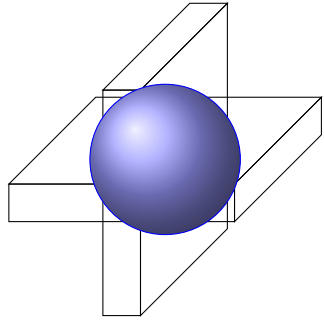

For $x\in N_{\De}(F')$, we have
\begin{equation}\label{hardhighlow}
    \de^{4\e}\De^{-d_{m,n}}  \lesssim f(x)\le |f_{\textup{high}}(x)|+|f_{\textup{low}}(x)|.
\end{equation}
We will show that the high part dominates for $x\in N_{\De}(F')$, i.e., $|f_{\textup{high}}(x)|\gtrsim \de^{4\e}\De^{-d_{m,n}}$.
It suffices to show 
\begin{equation}\label{pointwise}
    |f_{\textup{low}}(x)|\le C^{-1}\de^{4\e}\De^{-d_{m,n}},
\end{equation}
for a large constant $C$.

Recall that $f_{\textup{low}}=\sum_{V\in G'} f_V\ast\eta_{\textup{low}}^\vee$. Since $\eta_{\textup{low}}$ is a bump function on $B^n (0,(K\De)^{-1})$, we see that $\eta_{\textup{low}}^\vee$ is an $L^1$-normalized bump function essentially supported in $B^n(0,K\De)$. Let $\chi(x)$ be a positive function $=1$ on $B^n(0,K\De)$ and decays rapidly outside $B^n(0,K\De)$. We have
$$ |\eta_{\textup{low}}^\vee|\lesssim \frac{1}{|B^n(0,K\De)|}\chi. $$
Therefore,
\begin{equation}\label{bumpest}
    |f_{\textup{low}}(x)|\lesssim \sum_{V\in G'}\sum_{T\in\T_V} \psi_T * \frac{1}{|B^n(0,K\De)|}\chi(x)\lesssim \sum_{V\in G'}\sum_{T\in\T_V} K^{-m}\chi_{T_K}(x).
\end{equation}
Here, each $T_K$ is a plank of dimensions \[\underbrace{K\De \times K\De \times \dots \times K\De}_{m \text{ times}} \times \underbrace{1\times 1\times \dots \times 1}_{n-m \text{ times}}\]
which is the $K$-thickening of the $\De\times \dots \times \De$-side of $T$, and $\chi_{T_K}$ is a bump function $=1$ on $T_K$ and decays rapidly outside $T_K$. The rapidly decaying tail is negligible, so we can think of each $\chi_{T_K}$ as the indicator function of $T_K$. For a fixed $V\in G'$, we note that $\{T:T\in\T_V\}$ are orthogonal to $V$. Therefore, if we let $P_{K\De}$ be a plank of dimensions
\[\underbrace{K\De \times K\De \times \dots \times K\De}_{m \text{ times}} \times \underbrace{1\times 1\times \dots \times 1}_{n-m \text{ times}}\]
that is orthogonal to $V$ and contains $x$, then by condition (i),
$$ \sum_{T\in\T_V}\chi_{T_K}(x)\lesssim \#\{T\in\T_V: T\subset P_{K\De}\}\lesssim K^a, $$
where the last inequality is by the $a$-dimensional condition of $\T_V$.
Plugging this back into \eqref{bumpest}, we obtain
$$ |f_{\textup{low}}(x)|\lesssim K^{a-m} \#G'\lesssim K^{a-m}\De^{-d_{m,n}}. $$
($G'$ is a $(\De,d_{m,n})$-set, so $\#G'\lesssim \De^{-d_{m,n}}$.)

Noting that $a<m$, we may choose $K\sim_{a,m} \de^{-O_{a,m}(\e)}$ so that \eqref{pointwise} holds.
Plugging back to \eqref{tdimlower}, we have
$$ \De^{n-a-2d_{m,n}}\lessapprox \int |f_{\textup{high}}|^2 = \int \left|\sum_{V\in G'} \wh{f_V} \eta_{\textup{high}}\right|^2. $$
Here, $A\lessapprox B$ means $A\lesssim \de^{-O(\e)}B$. It is good to mention that since $\De\le \de^{\eta/a}$, by choosing $\e$ small enough depending on $\eta, a$, we have that $K$ is much smaller than $\De^{-1}$.

We use the following lemma to estimate the overlap of $\{\supp(\wh{f_V}\eta_{\textup{high}})\}_{V\in G'}$, or more precisely $\{D_V\setminus B^n(0,(K\De)^{-1})\}_{V\in G'}$. After rescaling $x\mapsto \De x$, each $D_V$ becomes a $1\times \dots\times 1\times \De\times\dots\times \De$-plank with $m $ many $1$'s and $(n-m)$ many $\De$'s in the expression. We denote this rescaled plank by $P_V$. We can see that $P_V$ is morally $N_\De(V)\cap B^n(0,1)$. It is harmless to just assume 
\[P_V=N_\De(V)\cap B^n(0,1).\]
We also see that after rescaling, $D_V\setminus B^n(0,(K\De)^{-1})$ becomes $P_V\setminus B^n(0,K^{-1})$. We will bound the overlaps of $\{P_V\setminus B^n(0,K^{-1})\}$ where $\{P_V\}$ are essentially distinct.

\begin{lemma}\label{lemgrass}
$\{P_V\setminus B^n(0,K^{-1})\}_{V\in G'}$ is $\lesssim K^{O(1)} \De^{-\dim (G(m-1,n-1))}$-overlapping.
\end{lemma}
\begin{proof}
We will estimate the number of overlaps at the point $\boldsymbol{\xi}_0=(0,\dots, 0 ,\lambda)$ with $\lambda\in[K^{-1},1]$. We just need to show that the number of planks $P_V$ that pass through $0$ and  $\boldsymbol{\xi}_0$ is $\lesssim K^{O(1)} \De^{-\dim (G(m-1,n-1))}$.

We first talk about some properties for the smooth manifold $G(m,n)$. For $V_1, V_2\in G(m,n)$, define $d(V_1,V_2)=\|\pi_{V_1}-\pi_{V_2}\|$. Then $d(\cdot,\cdot)$ gives a metric on $G(m,n)$. We need another characterization for this distance. Define $\rho(V_1,V_2)$ to be the smallest number $\rho$ such that $B^n(0,1)\cap V_1\subset N_{\rho}(B^n(0,1)\cap V_2)$. We claim that $\rho(V_1,V_2)\sim d(V_1,V_2)$. Suppose $B^n(0,1)\cap V_1\subset N_{\rho}(B^n(0,1)\cap V_2)$, then for any $v\in\R^n$ we have
\[ |\pi_{V_1}(v)-\pi_{V_2}(v)|\lesssim \rho |v|, \]
which implies $d(V_1,V_2)\lesssim \rho$. On the other hand, if for any $|v|\le 1$ we have 
\[|\pi_{V_1}(v)-\pi_{V_2}(v)|\le d|v|, \]
then we obtain that $\pi_{V_1}(v)\subset N_{Cd}(B^n(0,1)\cap V_2)$. Letting $v$ range over the unit ball in $V_1$, we get $B^n(0,1)\cap V_1\subset N_{Cd}(B^n(0,1)\cap V_2)$.

Consider the $\wt G=\{W\in G(m,n): 0,\boldsymbol{\xi}_0\in W\}$ which is a submanifold of $G(m,n)$. $\wt G$ is the set of $m$-subspaces that contain the $n$-th axis. Actually, it is not hard to see that $\wt G$ is isomorphic to $G(m-1,n-1)$. We return back to $P_V=N_\De(V)\cap B^n(0,1)$.
We make the following geometric observation: if $\boldsymbol{\xi}_0\subset P_V$, then there exists $W\in \wt G$ so that $W\cap B^n(0,1)\subset N_{CK\De}(V)\cap B^n(0,1)$. Recall the length of $\boldsymbol{\xi}_0$ is $\lambda\in [K^{-1},1]$, so the angle between $\boldsymbol{\xi}_0$ and $V$ is $\lesssim \De K$. Therefore the unit vector $\lambda^{-1} \boldsymbol{\xi}_0$ is contained in $N_{CK\De}(V)\cap B^n(0,1)$. It suffices to find an $m$-dimensional space $W$ such that $\lambda^{-1}\boldsymbol{\xi}_0\in W$ and $W\cap B^n(0,1)\subset N_{10CK\De}(V)\cap B^n(0,1)$. Let $v$ be the projection of $\lambda^{-1} \boldsymbol{\xi}_0$ onto $V$, then the angle between $v$ and $\lambda ^{-1}\boldsymbol{\xi}_0$ is $\lesssim \De K$. Imagine we choose a family of vectors $v(\theta), \theta\in [0,\De K]$ so that $v(0)=v$, $v(\De K)=\lambda^{-1} \boldsymbol{\xi}_0$ and also $|v(\theta_1)-v(\theta_2)|\lesssim |\theta_1-\theta_2|.$ Actually, we can choose them so that $v(\theta)$ lies on the line segment connecting $v(0)$ and $v(\De K)$.  Starting with $\theta=0$, we choose the $m$-dimensional space $V(0)=V$ so that $v(0)\in V(0)$. When $\theta$ changes we get a family of $m$-dimensional subspaces $V(\theta)$ by rotating $V(0)=V$ so that $v(\theta)\in V(\theta)$. When $\theta$ changes from $0$ to $\De K$, we see we rotate $V$ to another space $W=V(\De K)$ within angle $\lesssim \De K$. Therefore we find the $W$.

We proved that there exists $W\in \wt G$ so that $W\cap B^n(0,1)\subset N_{CK\De}(V)\cap B^n(0,1)$. By the comparability of the metric discussed in the previous two paragraphs, we see that $d(V,\wt G)\lesssim \De K$. In other word, those $V\in G(m,n)$ satisfying $\boldsymbol{\xi}_0\in P_V$ is contained in the $C\De K$-neighborhood of $\wt G$ in $G(m,n)$. We denote this neighborhood by $N_{C\De K}(\wt G)$.  
Noting that $\wt G$ is submanifold of dimension $\dim(G(m-1,n-1))=(m-1)(n-m)$ and $G'$ is a $\De$-separated subset of $G(m,n)$, we get the number of overlaps of $G'$ at $\boldsymbol{\xi}_0$ is
\[  \lesssim \textup{measure}\big(N_{C\De K}(\wt G)\big) /\De^{\dim(G(m,n))}\sim K^{O(1)}\De^{-\dim(G(m-1,n-1))}. \]
Note that we use a simple fact: If $\wt M$ is an $m$-dimensional smooth submanifold of  the $n$-dimensional manifold $M$, then 
\[\textup{measure}\big( N_r(\wt M) \big)\lesssim_{M,\wt M} r^{n-m}, \]
for $0\le r\le 1$.
\end{proof}

We are now able to find an upper bound to the high part of the integral. We have
\begin{align*}
    \De^{n-a-2d_{m,n}}\lessapprox \int |f_{\textup{high}}|^2= \int |\wh{f}_{\textup{high}}|^2 
    \lessapprox  \De^{-\dim (G(m-1,n-1))}\sum_{V\in G'} \int |\eta_{\textup{high}}\wh{f}_V|^2 
\end{align*}
by Lemma \ref{lemgrass}. Since $|\eta_{\textup{high}}|\lesssim 1$ and the planks in $\T_V$ (for a fixed $V$) are essentially disjoint, we have

\begin{align*}
    \int |\eta_{\textup{high}}\wh{f}_V|^2\lesssim \sum_{V\in G'} \int |f_V|^2
    \lesssim \sum_{V\in G'} \sum_{T\in \T_V} \int|\psi_{T}|^2 
    \le (\#G')(\#\T_V)\De^m 
     \lesssim \de^\eta\De^{-d_{m,n}-a+m}.
\end{align*}
Here we have a factor $\de^\eta$ because of the upper bound \eqref{radnote}. And we remark that $\de^\eta$ is quite important to get a contradiction.

Combining everything and noting that $\dim (G(m-1,n-1))=(m-1)(n-m)$, we have that 
\[
   1 \lessapprox  \de^\eta.
\]
Unwrapping the notation, we get
\[ 1\lesssim \de^{-O(\e)+\eta}. \]
This is impossible if we choose $\de, \e$ small enough depending on $\eta$. We get a contradiction.

\end{proof}

%%%%%%%%%%%%

%%%%%%%%%%%%%

\begin{proposition}\label{2prop}
Let $0<\si <n-1$, $a\in (n-1,n]$ and $t>\max\{n-1+\si - a, 0\}$. Let $0<\eta<1$. Then for $\e$ and $\de$ small enough depending on $\si,t$, and $\eta$, we have the following result.  

Let $E,F\subset B^n(0,1)$ so that $E$ is a $(\de,t)$-set with $\#E\gtrsim \de^\e \de^{-t}$ and $F$ satisfies $\cH^a_{\de,\infty}(F)\gtrsim \de^\e$, $\#F\lesssim \de^{-a}$. (We remark that we did not assume $F$ is $\de$-separated.) We also assume: each of $E$ and $F$ lies in a ball of radius $1/1000$ and $\dist(E,F)\ge 1/2$.
Then, there exists $y\in E$ such that for all $F'\subset F$ with $\cH^a_{\de,\infty}(F')\ge \de^{\e}$, we have
\[
\mathcal H^\si_{\de,\infty}(\pi_y(F')) > \de^{\eta}.
\]
\end{proposition}

\begin{proof}
Since $n-1+\si-a<\si$, it suffices to prove the proposition for $t< \si$.
Assume for the sake of contradiction that for all $y\in E$ there exists $F_y\subset F$ with $\cH^a_{\de,\infty}(F_y)\ge \de^{\e}$ such that 
\[\cH^\si_{\de,\infty}(\pi_y(F_y))\le\de^\eta.\]

We first reduce $F$ to a $(\de,a)$-set. The algorithm goes as follows. By the condition that $\cH^a_{\de,\infty}(F)\gtrsim \de^\e$ and Lemma \ref{usefullemma2.5}, we can find a $(\de,a)$-set $F_1\subset F$ with $\#F_1\gtrsim \de^{-a+2\e}$. We look at $F\setminus F_1$. If $\cH^a_{\de,\infty}(F\setminus F_1)\le \de^{2\e}$, we stop; If $\cH^a_{\de,\infty}(F\setminus F_1)\ge \de^{2\e}$, we find a $(\de,a)$-set $F_2\subset F\setminus F_1$ with $\#F_2\gtrsim \de^{-a+2\e}$. Repeating the algorithm until we stop, we obtain a decomposition
\[ F=\left(\bigsqcup_{i=1}^{N}F_i\right) \sqcup F_0, \]
where each $F_i\ (1\le i\le N)$ is a $(\de,s)$-set with cardinality $\gtrsim \de^{-a+2\e}$, and $F_0$ satisfies $\cH^a_{\de,\infty}(F_0)\le \de^{2\e}$. We also see that $N\lesssim \#F/\de^{-a+2\e}\lesssim \de^{-2\e}$. 

For any $y\in E$, we have $\de^\e\le \sum_{i=0}^N \cH^a_{\de,\infty}(F_y\cap F_i)\le  \de^{2\e}+ \sum_{i=1}^N \cH^a_{\de,\infty}(F_y\cap F_i)$. By pigeonholing, there exists $i=i(y)$ such that $\cH^a_{\de,\infty}(F_y\cap F_i)\gtrsim \de^{3\e}$.
By another pigeonholing, there exists $i\in [1,N]$, such that
\[ \#\{y\in E: i(y)=i\}\gtrsim \de^{2\e}\#E. \]
For simplicity, we will still use the old notation.
We replace $E$ by $\{y\in E: i(y)=i\}$, $F$ by $F_i$, $F_y$ by $F_y\cap F_i$, and $\e$ by $\e/10$. Then, $E$ is still a $(\de,t)$-set with $\#E\gtrsim \de^\e\de^{-t}$; $F$ is a $(\de,a)$-set with $\#F\gtrsim \de^\e\de^{-a}$; $F_y\subset F$ and $\#F_y\gtrsim \de^\e \#F$ for each $y\in E$ (since $\cH^a_{\de,\infty}(F_y)\gtrsim \de^{3\e/10}$ implies $\#F_y\gtrsim \de^{-a+\e}\gtrsim \de^\e\#F$); moreover, 
\[\cH^\si_{\de,\infty}(\pi_y(F_y))\le \de^\eta.\] 
We will derive a contradiction.

By the definition of $\cH^\si_{\de,\infty}$, we can find a covering of $\pi_y(F_y)$ by dyadic caps $\{D\}$ in $\S^{n-1}$, so that $\cH^\si_{\de,\infty}(\pi_y(F_y))=\sum_D r(D)^\si.$ For each such $D$, consider $\pi_y^{-1}(D)\cap \big(B^n(y,2)\setminus B^n(y,1/4)\big)$. It is roughly a tube of length $\sim 1$ and radius comparable to the radius of $D$. By the separation of $E, F$ and noting $E, F$ are contained in $B^n(0,1)$. We see that the tubes obtained in this way form a covering of $F_y$:
\[ F_y\subset \bigsqcup_{\de\le \De\le 1}\bigcup_{T\in\T_{y,\De}}T. \]
Here, $\T_{y,\De}$ consists of essentially disjoint tubes of radius $\sim \De$ and length $\sim 1$, and they all point to the point $y$. (See Figure \ref{radialprojprop} for the configuration of these tubes.)

\begin{figure}
\begin{tikzpicture}[x=0.75pt,y=0.75pt,yscale=-1,xscale=1]
%uncomment if require: \path (0,271); %set diagram left start at 0, and has height of 271

%Shape: Ellipse [id:dp015175412233626062] 
\draw   (242,73) .. controls (242,52.01) and (277.82,35) .. (322,35) .. controls (366.18,35) and (402,52.01) .. (402,73) .. controls (402,93.99) and (366.18,111) .. (322,111) .. controls (277.82,111) and (242,93.99) .. (242,73) -- cycle ;
%Shape: Ellipse [id:dp9041880446429325] 
\draw   (243,192) .. controls (243,171.01) and (278.82,154) .. (323,154) .. controls (367.18,154) and (403,171.01) .. (403,192) .. controls (403,212.99) and (367.18,230) .. (323,230) .. controls (278.82,230) and (243,212.99) .. (243,192) -- cycle ;
%Shape: Circle [id:dp5596862970264425] 
\draw  [dash pattern={on 4.5pt off 4.5pt}] (255,73) .. controls (255,59.19) and (266.19,48) .. (280,48) .. controls (293.81,48) and (305,59.19) .. (305,73) .. controls (305,86.81) and (293.81,98) .. (280,98) .. controls (266.19,98) and (255,86.81) .. (255,73) -- cycle ;
%Shape: Rectangle [id:dp5436700681431033] 
\draw   (274,98.33) -- (286,98.33) -- (286,240) -- (274,240) -- cycle ;
%Shape: Rectangle [id:dp5474205934858001] 
\draw   (342.08,227.02) -- (330.8,231.46) -- (286,98.33) -- (297.28,93.9) -- cycle ;
%Shape: Rectangle [id:dp058313028278387025] 
\draw   (280,98.33) -- (291.82,96.08) -- (314.57,235.59) -- (302.75,237.85) -- cycle ;
%Flowchart: Summing Junction [id:dp36616720138444125]

\foreach \Point in {(278,73), (280,213.5), (305,215.5), (280,188.5), (300,189.5), (327.5,204)}{
    \node at \Point {\textbullet};
}

% Text Node
\draw (408,47.4) node [anchor=north west][inner sep=0.75pt]    {$E$};
% Text Node
\draw (409,200.4) node [anchor=north west][inner sep=0.75pt]    {$F$};
% Text Node
\draw (282,67.4) node [anchor=north west][inner sep=0.75pt]    {$y$};

\end{tikzpicture}
\caption{$\T_{y,\De}$ in the radial projection}
\label{radialprojprop}
\end{figure}

$\T_{y,\De}$ satisfies the $\si$-dimensional spacing condition (inherited from $\{D\}$): For $\De\le r\le 1$, if $T_r$ is a tube of radius $r$ length $1$ that passes through $y$, then $T_r$ contains $\lesssim  (r/\De)^\si$ many tubes from $\T_{y,\De}$.
Since $\cH^\si_{\de,\infty}(\pi_y(F_y))\le \de^\eta$, we have
\begin{equation}\label{radnote6}
    \#\T_{y,\De}\lesssim \de^\eta \De^{-\si}.
\end{equation}
We see that $\T_{y,\De}$ is non-empty only for $\De\le \de^{\eta/\si}$.

Next, we will apply a standard pigeonhole argument to find a scale $\De$. Note that 
\[ F_y\subset\bigsqcup_{\de\le\De\le\de^{\eta/a}}\bigcup_{T\in\T_{y,\De}}T. \]
For each $y\in E$, we can find a dyadic $\De(y)\in [\de,\de^{\eta/a}]$ so that
\begin{equation}\label{pig06}
    \#(F_y\cap\bigcup_{T\in\T_{y,\De(y)}}T)\gtrsim |\log\de|^{-1}\#F_y\gtrsim \de^{\e}\#F.
\end{equation}
Define $E_\De=\{y\in E: \De(y)=\De\}$. We see that
\[ E=\bigcup_{\de\le\De\le\de^{\eta/a}}E_\De. \]
By pigeonholing again, we can find a scale $\De$, such that
\begin{equation}\label{pig26}
    \#E_\De\gtrsim \de^\e\#E.
\end{equation}
We fix this $\De$. Noting that $E$ is a $(\de,t)$-set with $\#E\gtrsim \de^\e\de^{-t}$, we have that $E_\De$ is also a $(\de,t)$-set with $\#E_\De\gtrsim\de^{2\e}\de^{-t}$. By lemma \ref{usefullemma3}, we can find a subset $E'$ of $E_\De$ so that $E'$ is a $(\De,t)$-set with $\#E'\gtrsim\de^{2\e}\De^{-t}$. From \eqref{pig06}, we have for any $y\in E'$ that
\begin{equation}\label{pig16}
    \#(F\cap \bigcup_{T\in \T_{y,\De}}T)\gtrsim \de^{\e}\#F\gtrsim \de^{2\e-a} .
\end{equation}

Next, we consider the set
\[ S:=\{ (x,y)\in F\times E': x\in \bigcup_{T\in \T_{y,\De}}T \}. \]
Define the sections of $S$:
\[ S_x:=\{y\in E': (x,y)\in S\},\ \ \ S^y:=\{x\in F: (x,y)\in S\}. \]
By \eqref{pig16}, we have $\#S^y\gtrsim \de^{\e}\#F$ for $y\in E'$. Then we have
\begin{equation}\label{pig36}
    \#S=\sum_{y\in E'} \#S^y\ge C^{-1} \de^{\e} \#E'\#F.
\end{equation}
Since
\[ \#\{(x,y)\in S: \#S_x\le  (2C)^{-1}\de^{2\e}\#E' \}\le (2C)^{-1} \de^{2\e} \#E'\#F\le \frac{1}{2} \#S,  \]
we have
\[ \#\{(x,y)\in S: \#S_x\ge  (2C)^{-1}\de^{2\e}\#E' \}\gtrsim \de^{\e} \#E'\#F. \]
The inequality above implies
\[ \#\{ x\in F: \#S_x \ge  (2C)^{-1}\de^{2\e}\#E' \}\gtrsim \de^{\e} \#F. \]
We define
\begin{equation}\label{deffdelta6}
     F_\De:= \{ x\in F: \#S_x \ge  (2C)^{-1}\de^{2\e}\#E' \}.
\end{equation}
Noting that $F_\De$ is a $(\de,a)$-set with $\#F_\De\gtrsim \de^{2\e-a}$, by Lemma \eqref{usefullemma3}, we can find a subset $F'\subset F_\De$ such that $F'$ is a $(\De,a)$-set with
\begin{equation}
    \#F'\gtrsim \de^{2\e}\De^{-a}.
\end{equation}

Let us summarize what we obtained.
We find a scale $\De\in[\de,\de^{\eta/\si}]$, a $(\De,t)$-set $E'\subset E$ with $\#E'\gtrsim \de^{2\e}\De^{-t}$, and a $(\De,a)$-set $F'\subset F$ with $\#F'\gtrsim \de^{2\e}\De^{-a}$, so that
\begin{enumerate}[label=(\roman*)]
    \item for each $y\in E'$, we have a set of tubes $\T_{y,\De}$ that satisfy the $\si$-dimensional spacing condition with $\#\T_{y,\De}\lesssim \de^\eta\De^{-\si}$ (see paragraph before \eqref{radnote6}),
    \item  each $x\in F'$ is contained in $\gtrsim \de^{2\e}\#E'\gtrsim \de^{4\e}\De^{-t}$ tubes from $\bigcup_{y\in E'}\T_{y,\De}$ (see \eqref{deffdelta6}).
\end{enumerate}
We see that we have reduced the problem to the following lemma, and we will get a contradiction from the following lemma.

\end{proof}

\begin{lemma}\label{discorp}
Let $0<t<\sigma <n-1$, $a\in (n-1,n]$. Let $0<\de\le\De\le \de^{\eta/\si}$, $\e>0$, where $\de,\e$ are small enough depending on $\eta,t,\si,a$.  Let $E,F\subset B^n(0,1)$ be non-empty $\De$-separated sets where
\begin{enumerate}
    \item $E$ is a $(\De,t)$-set with cardinality $\#E \gtrsim \De^{-t}\de^\e$,
    \item $F$ is a $(\De,a)$-set with cardinality $\#F \gtrsim \De^{-a}\de^\e$,  
    \item  each of $E$ and $F$ lies in a ball of radius $1/1000$ and $\dist(E,F)\ge 1/2$.
\end{enumerate}
For all $y\in E$, we assume there exists a collection of $\De$-tubes $\T_{y}$, such that 
\begin{enumerate}
    \item each $T\in\T_y$ is of form $\pi_y^{-1}(C)\cap\{x\in\R^n:1-\frac{1}{100} \le |x-y|\le 1\}$ for some dyadic $\De$-cap $C\subset \S^{n-1}$,
    \item $\T_{y}$ is a $(\De, \sigma)$-set of tubes with cardinality $\#\T_y\lesssim \de^\eta\De^{-\si}$,
    \item and for all $x\in F$, $\#\{y\in E: \exists T\in \T_{y} \text{~such that~} x\in T\}\gtrsim \De^{-t}\de^\e$.
\end{enumerate}
Then, \[
\de^{O(\e)}\Delta^{-t}\lesssim \de^{\frac{\si-t}{\si}\eta}\Delta^{-(n-1) -\sigma +a},
\]
which implies that $t\le n-1+\si-a$ (if $\de$ is small enough and $\e$ is very small depending on $\eta, t, \si$). This contradicts the condition in Proposition \ref{2prop}.
\end{lemma}

\begin{proof}
We will modify $\T_y$ a little bit. Since we will consider the interplay among $\{\T_y\}_{y\in E}$, we want to make the comparable tubes to be exactly the same. Note that $F$ is contained in a ball $B_{1/1000}$ of radius $1/1000$. We choose a set of $\de/100$-separated directions in $\S^{n-1}$, denoted by $\Theta=\{\theta\}$. For each direction $\theta\in\Theta$, we choose $\T_\theta$ to be a set of $100\de$-tubes that point to the direction $\theta$ and form a finitely overlapping covering of $B_{1/1000}$. Denote $\T=\cup_\theta \T_\theta$. If $\{\T_\theta\}_{\theta\in\Theta}$ are chosen properly, then for any $\de$-tubes $T$, there exists $T'$ in some $\T_\theta$ such that $T\cap B_{1/1000}\subset T'$. We modify every $\T_y$ in this way: we replace every $T$ in $\T_y$ by a tube $T'$ in some $\T_\theta$ such that $T\subset T'$. After replacement of $\T_y$, we still denote it by $\T_y$, but now $\T_y\subset \T$. Also, the new $\T_y$ inherits the properties of the old $\T_y$: $\T_y$ is a $(\De,\si)$-set with $\#\T_y\lesssim \de^\eta\De^{-\si}$; for any $x\in F$, $\#\{y\in E: \exists T\in \T_{y} \text{~such that~} x\in T\}\gtrsim \De^{-t}\de^\e$. After the modification, $\T_y$ is still a truncated bush centered at $y$.

Fix a $y\in E$. For any $T \in \T_{y}$, choose a bump function $\psi_{T}$ such that $\psi_{T}\geq 1$ on $T$, $\psi_{T}$ decays rapidly outside of $T$, and $\supp ~\wh \psi_{T}$ is contained in the dual rectangle of $T$ which is a $\De^{-1}\times\dots\times\De^{-1}\times 1$-slab. Define 
\[
f_y = \sum_{T\in \T_{y}} \psi_{T} \hspace{.25cm}\text{and}\hspace{.25cm} f = \sum_{y\in E} f_y.
\]
Then, for $x\in N_{\De}(F)$, $f(x) \gtrsim \#\{y\in E: \exists T\in \T_{y} \text{~such that~} x\in T\}\gtrsim \De^{-t}\de^\e$ by assumption. Therefore, 
\begin{equation} \label{liu3}
    \de^{O(\e)}\De^{-2t-a+n} \lesssim \de^{2\e}\De^{-2t}(\# F)\De^n \lesssim \int_{N_\De(F)} |f|^2.
\end{equation}
We will use the high-low method. Let $\eta_{\textup{low}}(\xi)$ be a smooth bump function on $B^n(0,(K\De)^{-1})$ and $\eta_{\textup{high}}=1-\eta_{\textup{low}}$.
We will choose $K\sim \de^{-O(\e)}$. Define $f_{\textup{low}}=\eta_{\textup{low}}^\vee* f$ and $f_{\textup{high}}=\eta_{\textup{high}}^\vee*f$.

For $x\in N_{\De}(F)$, we have
$$ \De^{-t}\de^\e \lesssim f(x)\le |f_{\textup{low}}(x)|+|f_{\textup{high}}(x)|. $$
We claim that
$$ |f_{\textup{low}}(x)|\lesssim K^{\si-(n-1)}\#E\le C^{-1}\De^{-t}\de^{\e}, $$
if $K\sim \de^{-O(\e)}$ is properly chosen. 
To show the claim, we write
\[ |f_{\textup{low}}(x)|\le \sum_{y\in E}|\eta_{\textup{low}}^\vee| *f_y(x)\le \sum_{y\in E}\sum_{T\in\T_y}|\eta_{\textup{low}}^\vee| *\psi_T(x). \]
Note that $|\eta_{\textup{low}}^\vee|(x)\lesssim  (K\De)^{-n}\chi(x)$, where $\chi(x)$ is a positive function $=1$ on $B^n(0,K\De)$ and decays rapidly outside $B^n(0,K\De)$. Therefore,
\[ |\eta_{\textup{low}}^\vee| *\psi_T(x)\lesssim K^{-(n-1)}\chi_{KT}(x), \]
where $\chi_{KT}(x)=1$ on $KT$ and decays rapidly outside $KT$. Since $\T_y$ is a $(\De,\si)$-set, we have for $x\in F$,
\[ \#\{T\in \T_y: x\in 100KT \}\lesssim K^\si. \]
Therefore, $\sum_{T\in\T_y}|\eta_{\textup{low}}^\vee| *\psi_T(x)\lesssim K^{\si-(n-1)}$. Summing over $y\in E$, we prove the claim.

Therefore, we have $|f(x)|\lesssim |f_{\textup{high}}(x)|$ on $N_{\De}(F)$.

We have
$$ \int_{N_\De(F)}|f|^2\lesssim \int |f_{\textup{high}}|^2. $$

Here is where things become a little more different than the high-low argument in the proof of Proposition \ref{1prop}. A tube may belong to many different $\T_y$. For each $T\in \T$, define 
\[
n_T := \#\{y\in E\mid T\in \T_{y}\}.
\]
$n_T$ can be $0$, which means $T\notin \T_y$ for any $y\in E$.
We have
\[ \int |f_{\textup{high}}|^2 = \int \left|\sum_{T\in \T} n_T\cdot \psi_{T,\textup{high}}\right|^2. \]
Here, $\psi_{T,\textup{high}}=\eta_{\textup{high}}^\vee*\psi_T$.
If $T\in \T_\theta$, let $S_\theta$ be the slab centered at the origin, of dimensions $\De^{-1}\times\dots\times \De^{-1}\times 1$, which is the dual of $T$. We also see that $S_\theta$ is the dual of any $T\in\T_\theta$. Now we have
\[ \supp (\wh \psi_{T,\textup{high}})\subset S_\theta\setminus B^n(0,K\De). \]

Applying Lemma \ref{lemgrass} at the special case that $m=n-1$ we see that $\{S_\theta\setminus B^n(0,K\De)\}_{\theta\in\Theta}$ are $\lesssim K^{O(1)}\De^{-\dim(G(n-2,n-1))}$-overlapping. We do the following estimate

\begin{align*}
    \int |f_{\textup{high}}|^2 = \int \left|\sum_{T\in \T} n_T\cdot \psi_{T,\textup{high}}\right|^2=\int \left| 
 \sum_{\theta\in\Theta}\sum_{T\in \T_\theta} n_T\cdot \psi_{T,\textup{high}}\right|^2\\
 \lesssim \de^{-O(\e)}\De^{-(n-2)}\sum_{\theta\in\Theta}\int|\sum_{T\in \T_\theta} n_T \psi_{T,\textup{high}}|^2\\
\lesssim \de^{-O(\e)}\De^{-(n-2)}\sum_{T\in \T} n_T^2\int |\psi_{T,\textup{high}}|^2\\
\lesssim\de^{-O(\e)}\De^{-(n-2)}\sum_{T\in \T} n_T^2\int |\psi_{T}|^2\lesssim \de^{-O(\e)}\De\sum_{T\in\T}n_T^2.
\end{align*}
In the second last row above, we use the fact that tubes in $\T_\theta$ are parallel and finitely overlapping, and hence the essential supports of $\{\psi_{T,\textup{high}}\}_{T\in\T_\theta}$, $\{KT\}_{T\in\T_\theta}$, are at most $K^{O(1)}$-overlapping. In the last row above, we use Young's inequality: $\int |\eta_{\textup{high}}^\vee*\psi_{T}|^2\lesssim (\int |\eta_{\textup{high}}^\vee|)^2 \int |\psi_T|^2\lesssim K^{O(1)} \int |\psi_T|^2$. 

We are going to find an upper bound to $\sum_{T\in \T} n_T^2$. The intuition is that $n_T=1$ for $T \in\cup_{y\in E} \T_y$, and $=0$ for other $T\in \T$. Therefore $\sum_{T\in\T} n_T^2=\sum_{T\in\T} n_T=\sum_{y\in E} \#\T_y\lesssim \# E \#\T_y\lesssim \De^{-t}\de^\eta\De^{-\si}$. We verify this intuition.

\begin{align*}
    \sum_{T\in \T} n_T^2 &= \sum_{T\in \T} \#\{y,y'\in E \mid T\in \T_y\cap \T_{y'}\}
    = \sum_{y\in E}\sum_{y'\in E} \#\{T\in \T\mid T\in \T_y\cap \T_{y'}\}
\end{align*}
Given that each $\T_y$ is a $(\De,\sigma)$-set, the above expression is bounded by
\begin{align*}
    \lesssim \sum_{y\in E}\sum_{y'\in E\setminus\{y\}} \min\left\{|y-y'|^{-\sigma},\#\T_y\right\} +\sum_{y\in E}\#\T_y.
\end{align*}
The second term is bounded by 
$$\sum_{y\in E}\#\T_y\lesssim \de^\eta \De^{-t-\si}.$$
For the first term, we have
\begin{align*}
    &\lesssim \sum_{y\in E} \sum_{k=0}^{\log_2 \De^{-1}} \sum_{|y-y'| \leq 2^{-k}} \min\{|y-y'|^{-\sigma},\de^\eta\De^{-\si}\} \\
    &\lesssim \sum_{y\in E} \sum_{k=0}^{\log_2 \De^{-1}}  \# \{y'\in E\cap B^n(y,2^{-k})\} \min\{2^{k\sigma},\de^\eta\De^{-\si}\}\\
    &\lesssim \De^{-t} \sum_{k=0}^{\log_2 \De^{-1}}(\De^{-1}2^{-k})^t \min\{2^{k\sigma},\de^\eta\De^{-\si}\}\\
    &= \De^{-t} \sum_{k=0}^{\log_2 \De^{-1}}\De^{-t} \min\{2^{k(\sigma-t)},\de^\eta\De^{-\si}2^{-kt}\}.
\end{align*}
When $2^{k(\sigma-t)}=\de^\eta\De^{-\si}2^{-kt}$ or equivalently $2^{k\si}=\de^\eta\De^{-\si}$, the value of ``min" dominates. The expression above is therefore bounded by
$\de^{\frac{\si-t}{\si}\eta}\De^{-t-\si}$.

Combining all the estimates, we have
\[
    \sum_{T\in \T} n_T^2 \lesssim (\de^{\frac{\si-t}{\si}\eta}+\de^\eta)\De^{-t-\si}.
\]
Plugging into \eqref{liu3}, we have 
\[
\de^{O(\e)}\De^{-t} \lesssim \de^{\frac{\si-t}{\si}\eta}\De^{-(n-1)-\si+a}.
\]
\end{proof}

We now prove Theorem \ref{discradial}.
\begin{proof}[Proof of Theorem \ref{discradial}]
We will show that the result holds for $\e\le \e_0(\eta,\si,a,t)$, $\de\le \de_0(\eta,\si,a,t)$, where $\e_0(\eta,\si,a,t),\de_0(\eta,\si,a,t)$ depend on Proposition \ref{1prop} and \ref{2prop}. The key idea is to project the sets to a lower dimensional subspace. Similar ideas has appeared in \cite{du2021improved}.

Since $k+\si-a<k$, we may assume $t<k+1$ so that we can apply Proposition \ref{1prop} with $(a,m)=(t,k+1)$.
We will apply Proposition \ref{2prop} with $n=k+1$. For our purpose, we determine the parameters of Proposition \ref{2prop} in advance.
For fixed $\eta$, we first choose small number $\e'$ so that Proposition \ref{2prop} holds for $\e=\e'$. Then let the parameter $\eta$ in Proposition \ref{1prop} be $\e'$. We choose $\e$ so that Proposition \ref{1prop} holds for this $\e$.

Recall the condition in Theorem \ref{discradial} that each of $E$ and $F$ lies in a ball of radius $1/1000$ and $\dist(E,F)\ge 3/4$. By the separation of $E$ and $F$,
we can find $\wt G\subset G(n,k+1)$ which has measure $\ge 10^{-10}$, such that any $V\in\wt G$ satisfies $$\dist(\pi_V(E),\pi_V(F))\ge \frac{1}{2}.$$
We choose $G$ to be a maximal $\de$-separated subset of $\wt G$. Then $G$ is a  
$(\de,d_{k+1,n})$-set  with $\# G\gtrsim \de^{-d_{k+1,n}}$.

By Proposition \ref{1prop},  there exists a subset $G_1\subset G$ with $\#G_1\gtrsim \de^{-d_{k+1,n}}$, so that for any $V\in G_1$ we have
\begin{equation}\label{bigproj1}
    \cH^a_{\de,\infty}(\pi_V(F'))>\de^{\e'},\ \ \ \textup{for~any~}F'\subset F\textup{~with~}\#F'\ge\de^\e\#F.
\end{equation}

Similarly, there exists $V\in G_1$, so that
\begin{equation}\label{bigproj2}
    \cH^t_{\de,\infty}(\pi_V(E))>\de^{\e'}.
\end{equation}
We just fix this $V$ for which \eqref{bigproj1} and \eqref{bigproj2} hold.

We are about to apply Proposition \ref{2prop} (with $\e=\e'$). 
By \eqref{bigproj2}, there exists a $(\de,t)$-set  $E_V\subset \pi_V(E)$ with $\#E\gtrsim \de^{\e'-t}$. 
We also check that $\pi_V(F)$ satisfies the requirement in Proposition \ref{2prop}: $\cH^a_{\de,\infty}(\pi_V(F))\gtrsim \de^{\e'}$, $\#\pi_V(F)\le \#F \lesssim \de^{-a}$.

We find a point $\wt y\in E_V$ such that: for all $\wt F\subset \pi_V(F)$ with $\cH^{a}_{\de,\infty}(\wt F)\ge \de^{\e'}$, we have
\begin{equation}\label{bigproj3}
    \cH^\si_{\de,\infty}\Big(\pi_{\wt y}(\wt F)\Big)>\de^\eta. 
\end{equation}

We use this property to finish the proof. We choose $y\in E$ so that $\pi_V(y)=\wt y$. We show that this $y$ satisfies the requirement in Theorem \ref{discradial}. For any $F'\subset F$ with $\#F'\ge \de^{\e}\#F$, by \eqref{bigproj1} we have $\cH^{a}_{\de,\infty}(\pi_V(F'))\ge\de^{\e'}$. Plug in $\wt F=\pi_V(F')$ into \eqref{bigproj3}:
$$ \cH^\si_{\de,\infty}\Big(\pi_{\wt y}\big(\pi_V(F')\big)\Big)>\de^\eta. $$
Note that
$$\cH^\si_{\de,\infty}\Big(\pi_{y}(F')\Big)\ge \cH^\si_{\de,\infty}\Big(\pi_{\wt y}\big(\pi_V(F')\big)\Big), $$
as any covering of $\pi_{y}(F')$ naturally gives rise to a covering of $\pi_{\wt y}\big(\pi_V(F')\big)$ by the separation of $E, F$.
Therefore, we have
$$ \cH^\si_{\de,\infty}\Big(\pi_{y}(F')\Big)>\de^\eta. $$
\end{proof}

\section{Liu's conjecture on radial projections}\label{sec3}
In this section, we prove Liu's conjecture (Theorem \ref{liusbound}). The idea is the same as in \cite{orponenshmerkin2022exceptional}, but we still provide full details to clarify the numerology since we are in higher dimensions.

We repeat Theorem \ref{liusbound} here.

\begin{theorem}
Given a Borel set $E\subset \R^n$, with $\dim E\in (k-1,k]$ for some $k\in\{1,\dots, n-1\}$, then 
\[
\dim \{x\in \R^n\setminus E \mid \dim (\pi_x(E))< \dim E\}\leq k.
\]
\end{theorem}
It suffices to prove
\begin{proposition}\label{proliu}
Given a Borel set $E_0\subset \R^n$, with $\dim E_0\in (k-1,k]$ for some $k\in\{1,\dots, n-1\}$, and $\tau_0>0$, then we have 
\[
\dim \{x\in \R^n\setminus E_0 \mid \dim (\pi_x(E_0))< \dim E_0-10\tau_0\}\leq k.
\]
\end{proposition}

Since the proof of this proposition is technical, we start with a heuristic proof. One of the key tool is Theorem \ref{thmradial}.
\begin{proof}[A heuristic proof of Proposition \ref{proliu}]
We just need to prove this for $\dim E_0<k$.
We set $s=\dim E_0$. 
Let
\[F_0=\{x\in \R^n\setminus E_0 \mid \dim (\pi_x(E_0))< \dim E_0-10\tau_0\}\]
By contradiction, we assume $t=\dim F_0>k$. Also, by passing to a subset of $F_0$,  we may assume $t\in(k,k+1)$. Now we let this $F_0$ be the set $A$ in Theorem \ref{thmradial}. Since $s<k$, we have that the $s$-exceptional 
$$ E_{s}(F)=\{y\in\R^n\setminus F_0: \dim(\pi_y(F_0))<s \} $$
has dimension $\le k+s-t<s=\dim E_0$. Subtracting this small exceptional part from $E_0$, we may pass to a subset of $E_0$ (still denoted by $E_0$) with the same dimension $s$ and satisfying 
$$ \dim(\pi_y(F_0))\ge s, $$
for any $y\in E_0$.

By $\de$-discretization, we may assume $F_0$ is a $t$-dimensional set of points and $E$ is an $s$-dimensional set of points. (Here, when we say $F_0$ is a $t$-dimensional set, it means that $F_0$ is a $(\de,t)$-set and $\#F_0\gtrsim \de^{-t}$). For each $x\in F_0$ and $y\in E_0$, we connect them by a $\de$-tube. Let $\T$ be the set of $\de$-tubes produced in this way. We also identify comparable tubes. Roughly speaking, we define
$$ \T:=\{T: T\textup{~connects~some~}x\in F, y\in E\}. $$
We also define $\T_x:=\{T\in\T: x\in T\}$ for $x\in F_0$, and $\T^y:=\{T\in\T:y\in T\}$ for $y\in E_0$.
By definition, we have $\dim(\pi_x(E_0))\le s-\tau_0$ for $x\in F_0$. This condition morally says that $\T_x$ is an $(s-\tau_0)$-dimensional set. Since the tubes in $\T_x$ are finitely overlapping at the portion away from $x$, we have
$$ \de^{-s}\le \#E_0\lesssim \sum_{T\in\T_x}\#(T\cap E_0). $$
Since $\#\T_x\le \de^{-s+\tau_0}$, we may morally assume $\#(T_x\cap E_0)\gtrsim \de^{-\tau_0/2}$ for any $T_x\in\T_x$. Morally, we may further assume for any $T\in\T$, we have $\#(T\cap E_0)\gtrsim \de^{-\tau_0/2}$. The condition $\dim(\pi_y(F_0))\ge s$ morally says that $\T^y$ is at least an $s$-dimensional set.

We consider the incidence between $E_0$ and $\T$. We will derive a contradiction by comparing the upper and lower bounds of $I(E_0,\T):=\{(y,T)\in E_0\times \T:y\in T \}$.
First, we have
$$ I(E_0,\T)=\sum_{T\in\T} \#(T\cap E_0)\gtrsim \#\T \de^{-\tau_0/2}. $$
For the upper bound of the incidence, we have
\begin{align*}
    I(E_0,\T)&=\sum_{T\in \T}\#(T\cap E_0)\le (\#\T)^{1/2}\left( \sum_{T\in\T} \#(T\cap E_0)^2 \right)^{1/2}\\
    &=(\#\T)^{1/2}\left( \sum_{y,y'\in E_0}\#\{T\in\T: y,y'\in T\} \right)^{1/2}\\
    &=(\#\T)^{1/2}\left( \sum_{y\in E_0}\sum_{y'\in E_0}\#\{T\in\T^y: y'\in T\} \right)^{1/2}.
\end{align*}
By the $s$-dimensional condition for $\T^y$, for $y\neq y'$ we have
\begin{align*}
    \#\{T\in\T^y: y'\in T\}\lesssim \left(\frac{\de}{|y-y'|}\right)^s\#\T^y . 
\end{align*}

Therefore, we have
\begin{align*}
    I(E_0,\T)&=(\#\T)^{1/2}\left( \sum_{y\in E_0}\sum_{y'\in E^*\setminus\{y\}}\#\{T\in\T^y: y'\in T\}+\sum_{y\in E_0}\#\T^y \right)^{1/2}\\
    &\lesssim (\#\T)^{1/2}\left( \sum_{y\in E_0}\sum_{y'\in E^*\setminus\{y\}}\left(\frac{\de}{|y-y'|}\right)^s\#\T^y+I(E_0,\T) \right)^{1/2}.
\end{align*}
Using that $E_0$ is an $s$-dimensional set, we have 
$$ \sum_{y'\in E_0\setminus\{y\}}\left(\frac{\de}{|y-y'|}\right)^s\lessapprox 1, $$
so we have
\begin{align*}
    I(E_0,\T)\lessapprox (\#\T)^{1/2}\big(1+I(E_0,\T)\big)^{1/2}.
\end{align*}
This means $I(E_0,\T)\lessapprox \#\T$, which contradicts the lower bound of $I(E_0,\T)$.
\end{proof}

\bigskip

We start the rigorous proof.
The proof is by contradiction to assume the set
\begin{equation}\label{defF}
    F_0=\{x\in \R^n\setminus E_0 \mid \dim (\pi_x(E_0))< \dim E_0-10\tau_0\}
\end{equation}
satisfies $t=\dim F_0>k.$
We will derive a contradiction through the following proposition and a standard reduction. It has the same idea in the proof that Theorem \ref{discradial} implies Theorem \ref{thmradial}. 

\begin{proposition}\label{discliu}
Let $k\in\{1,\cdots,n-1\}$.
Let $0<s<k$, $t>k$ and $\tau_0>0$. For $\e,\de$ small enough depending on $s,t,\tau_0$, the following holds. Let $E,F\subset B^n(0,1)$ be $(\de,s)$-set and $(\de,t)$-set, with $\#E\gtrsim \de^{-s+\e}$, $\#F\gtrsim \de^{-t+\e}$. We also assume that each one of $E, F$ is contained in a ball of radius $1/1000$, and $\dist(E,F)\ge 1/2.$ Then there exists $x\in F$
such that 
\begin{equation}\label{liubigpoint}
    |\pi_x(E')|_\de \ge \de^{-s+\tau_0},\ \ \textup{for~all~}E'\subset E \textup{~with~} \#E'\ge\de^\e\#E. 
\end{equation} 
\end{proposition}

\bigskip

\begin{proof}[Proof that Proposition \ref{discliu} implies Proposition \ref{proliu}]
We will do a same reduction as in the proof that Theorem \ref{discradial} implies Theorem \ref{thmradial}. Suppose $E_0$ is given in Proposition \ref{proliu}, and $F_0$ is given by \eqref{defF}.
Fix $\dim E_0-\tau_0<s_1<\dim E_0$. We can find $s_1$-dense points $y_1,y_2$ of $E_0$. Since our problem is scaling-invariant, we can assume $|y_1-y_2|=99/100$. We let $E_1=A\cap B_{1/1000}(y_1)$, $E_2=A\cap B_{1/1000}(y_2)$, and then $\dim(E_1),\dim(E_2)\ge s_1$. We only need to show for any ball $B_{1/1000}$ of radius $1/1000$, $F_0\cap B_{1/1000}$ has dimension $\le k$. Since $\dist(E_1,E_2)>98/100$, either $\dist(B_{1/1000},E_1)>3/4$ or $\dist(B_{1/1000},E_2)>3/4$. We may assume $\dist(B_{1/1000},E_1)>3/4$. We will show that the set
$$ F':=\{x\in B_{1/1000}:\dim(\pi_x(E_1))<\dim E_1-9\tau_0\}(\supset F_0\cap B_{1/1000}) $$
has dimension $\le k$.
From the reduction, these sets satisfy certain separation properties: 
\begin{align}
   \textup{each~one ~of~} E_1 \textup{~and~} F' \textup{~lies~in~some~ball~of~radius~} 1/1000,\\  E_1,F'\subset B^n(0,1),\ \  \dist(E_1,F')\ge 1/2.
\end{align}

We choose $t<\dim(F'),s=\dim(E_1)-\tau_0<\dim(E_1)$. Then $\cH^t_{\infty}(F')>0$, and by Frostman's lemma there exists a probability measure $\nu_{E_1}$ supported on $E_1$ satisfying $\nu_{E_1}(B_r)\lesssim r^s$ for any $B_r$ being a ball of radius $r$.
We can rewrite $F'$ as
\begin{equation}\label{Fprime}
    F'=\{x\in B_{1/1000}:\dim(\pi_x(E_1))< s-8\tau_0\}
\end{equation}

We only need to prove $t\le k$, since then we can send $t\rightarrow \dim(F')$. For the sake of contradiction, assume that $t> k$. 
Now we fix $t$, so we may assume $\cH^t_{\infty}(F')\sim 1$ is a constant.

Fix an $x\in F'$.
Using Lemma \ref{usefullemma} to $\pi_x(E_1)$, we obtain a set of dyadic caps $\cC_x=\bigsqcup_j \cC_{x,j}$ in $\S^{n-1}$ that cover $\pi_x(E_1)$. Here each $\cC_{x,j}$ is a set of $2^{-j}$-caps that satisfy the $(s-8\tau_0)$-dimensional condition (see Lemma \ref{usefullemma} (3)) because of $\dim (\pi_x(E_1))<s-8\tau_0$. Also, the radius of these caps is less than $\e_\circ$, which is any given small number. 

By the $(s-8\tau_0)$-dimensional condition of $\cC_{x,j}$, we have
\begin{equation}\label{2radcontr0.1}
    \# \cC_{x,j}\le 2^{j(s-8\tau_0)}. 
\end{equation}

\begin{figure}
\begin{tikzpicture}[x=0.75pt,y=0.75pt,yscale=-1,xscale=1]
%uncomment if require: \path (0,271); %set diagram left start at 0, and has height of 271

%Shape: Ellipse [id:dp015175412233626062] 
\draw   (242,73) .. controls (242,52.01) and (277.82,35) .. (322,35) .. controls (366.18,35) and (402,52.01) .. (402,73) .. controls (402,93.99) and (366.18,111) .. (322,111) .. controls (277.82,111) and (242,93.99) .. (242,73) -- cycle ;
%Shape: Ellipse [id:dp9041880446429325] 
\draw   (243,192) .. controls (243,171.01) and (278.82,154) .. (323,154) .. controls (367.18,154) and (403,171.01) .. (403,192) .. controls (403,212.99) and (367.18,230) .. (323,230) .. controls (278.82,230) and (243,212.99) .. (243,192) -- cycle ;
%Shape: Circle [id:dp5596862970264425] 
\draw  [dash pattern={on 4.5pt off 4.5pt}] (255,73) .. controls (255,59.19) and (266.19,48) .. (280,48) .. controls (293.81,48) and (305,59.19) .. (305,73) .. controls (305,86.81) and (293.81,98) .. (280,98) .. controls (266.19,98) and (255,86.81) .. (255,73) -- cycle ;
%Shape: Rectangle [id:dp5436700681431033] 
\draw   (274,98.33) -- (286,98.33) -- (286,240) -- (274,240) -- cycle ;
%Shape: Rectangle [id:dp5474205934858001] 
\draw   (342.08,227.02) -- (330.8,231.46) -- (286,98.33) -- (297.28,93.9) -- cycle ;
%Shape: Rectangle [id:dp058313028278387025] 
\draw   (280,98.33) -- (291.82,96.08) -- (314.57,235.59) -- (302.75,237.85) -- cycle ;
%Flowchart: Summing Junction [id:dp36616720138444125]

\foreach \Point in {(278,73), (280,213.5), (305,215.5), (280,188.5), (300,189.5), (327.5,204)}{
    \node at \Point {\textbullet};
}

% Text Node
\draw (408,47.4) node [anchor=north west][inner sep=0.75pt]    {$F'$};
% Text Node
\draw (409,200.4) node [anchor=north west][inner sep=0.75pt]    {$E_1$};
% Text Node
\draw (282,67.4) node [anchor=north west][inner sep=0.75pt]    {$x$};

\end{tikzpicture}
\caption{$\T_{x,j}$ in the radial projection}
\label{radialproj2}
\end{figure}
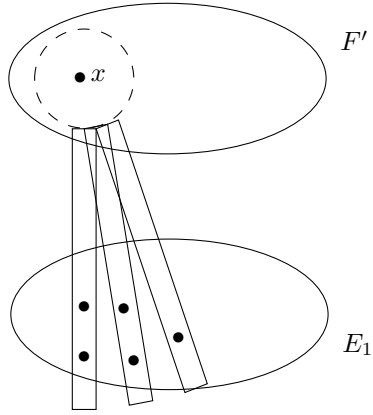

For each cap $C\in\cC_x$, consider $\pi_x^{-1}(C)\cap \{x\in\R^n:1-\frac{1}{100}\le|x-y|\le 1\}$ which is a tube. We obtain a collection of finitely overlapping tubes
$$ \T_x=\bigsqcup_j \T_{x,j} $$
that cover $E_1$ (see Figure \ref{radialproj2}).
Here, each tube has its coreline passing through $x$ and at distance $\sim 1$ from $x$. The tubes in $\T_{x,j}$ have dimensions $2^{-j}\times\dots\times2^{-j}\times 1$.
Also, $\T_{x,j}$ inherits the property \eqref{2radcontr0.1} from $\cC_{x,j}$:
\begin{equation}\label{radcontr2}
    \#\T_{x,j}\le 2^{j(s-8\tau_0)}
\end{equation}

For a fixed $x\in F'$, there exists a $j(x)\ge |\log_2\e_\circ|$ such that 
\begin{equation}\label{2radpigeon1}
    \nu_{E_1}\left(E_1\cap \bigcup_{T\in\T_{x,j(x)}}T\right)\ge \frac{1}{10j(x)^2}\nu_{E_1}(E_1)=\frac{1}{10j(x)^2}.
\end{equation}
We have a partition $F'=\bigsqcup_j F'_j$ where
$F'_j=\{x\in F': j(x)=j\}.$ We choose $j$ such that $\cH^t_\infty(F'_j)\gtrsim \frac{1}{j^2}$.
We let $\de=2^{-j}$. Note that $\de \leq \e_\circ$ by assumption. By Lemma \ref{frostmans}, there exists a subset $F''\subset F'_j$ which is a $(\de,t)$-set and $\# F''\gtrsim |\log\de|^{-2}\de^{-t}$. We use $\mu$ to denote the counting measure on $F''$.

Next, we consider the set $S=\{(y,x)\in E_1\times F'': y\in\bigcup_{T\in\T_{x,j}}T\}$. We also denote the $x$-section and $y$-section of $S$ by $S^x$ and $S_y$. (In Figure \ref{radialproj2}, $F'$ is drawn above $E_1$, so we use the convention that $x$ appears as the superscript in $S^x$.) By \eqref{2radpigeon1}, we have $\nu_{E_1}(S^x)\ge \frac{1}{10j(x)^2}$, so we have
\begin{equation}\label{2radpigeon3}
    (\nu_{E_1}\times \mu)(S)\ge \frac{1}{10j^2}\mu(F'').
\end{equation}
This implies 
\begin{equation}\label{2radpigeon4}
    (\nu_{E_1}\times \mu)\bigg(\Big\{(y,x)\in S: \mu (S_y)\ge \frac{1}{20j^2}\mu(F'')\Big\} \bigg)\ge \frac{1}{20j^2}\mu(F'').
\end{equation}
Therefore, we have
\begin{equation}\label{2radpigeon5}
    \nu_{E_1}\bigg(\Big\{y\in E_1: \mu (S_y)\ge \frac{1}{20j^2}\mu(F'')\Big\} \bigg)\ge \frac{1}{20j^2}\sim |\log\de|^{-2}.
\end{equation}
By Lemma \ref{frostmans3}, we can find a subset $E$ of $\Big\{y\in E_1: \mu (S_y)\ge \frac{1}{20j^2}\mu(F')\Big\}$, so that $E$ is a $(\de,s)$-set and $\#E\gtrsim |\log\de|^{-2}\de^{-s}$.

Hence,
\begin{equation}
   |\log\de|^{-2}\#E\#F'' \lesssim\#\left\{ (y,x)\in E\times F'': y\in \bigcup_{T\in\T_{x,j}} T \right\}=\sum_{x\in F''}\#\left\{y\in E: y\in \bigcup_{T\in\T_{x,j}} T \right\}.
\end{equation}
By pigeonholing, there exists a subset $F\subset F''$ with $\# F\gtrsim |\log\de|^{-2}\#F''\gtrsim \de^{\e/2} \de^{-t}$, so that for any $x\in F$:
$$ \#\{y\in E: y\in \bigcup_{T\in\T_{x,j}} T \}\gtrsim \de^\e\# E. $$
We set $E_x:=\{y\in E: y\in \bigcup_{T\in\T_{x,j}} T \}$.

Now we use Proposition \ref{discliu} to derive a contradiction. We just plug in the $E,F$ and check they satisfy the conditions of Proposition \ref{discliu}. Then it yields the existence of an $x\in F$ such that $|\pi_x(E')|_\de\ge \de^{-s+\tau_0}$, for any $E'\subset E$ with $\#E'\ge \de^\e\#E$. We just put $E'=E_x$, and see that $\de^{-s+\tau_0}\le |\pi_x(E_x)|_\de\lesssim \#\T_{x,j}\le \de^{-s+8\tau_0}$ by \eqref{radcontr2}.
This gives a contradiction if $\de$ is small enough depending on $\tau_0$.

\end{proof}

\begin{remark}
{\rm
\eqref{liubigpoint} roughly says there exists $x\in F$ such that $\dim(\pi_x(E))> \dim E-\tau_0$, contradicts the definition of $F_0$ in \eqref{defF}.
Throughout the proof of Proposition \ref{discliu}, we will use $x$ to denote points in $F$ and $y$ to denote points in $E$.
}
\end{remark}

It remains to prove Proposition \ref{discliu}. 

\subsection{Proof of Proposition \ref{discliu}}
We provide the full details for the proof of Proposition \ref{discliu}. We remark that the proof has the same idea as in \cite{orponenshmerkin2022exceptional}. We include here just for completeness.

In \cite{orponenshmerkin2022exceptional}, Orponen and Shmerkin derive their Corollary 4.5 from Proposition 4.2. By the same argument, we can derive the following corollary from Theorem \ref{discradial}.

\begin{corollary}\label{cor1}
Let $0\le \si\le s< k$, $t\in(k,k+1]$, $\eta>0$ very small, and $s>\max\{k+\si-t,0\}$. Then, for sufficiently small $\e,\de$ depending on $s,\si,t,\eta$, the following holds.

Let $E,F\subset B^n(0,1)$ be $(\de,s)$-set and $(\de,t)$-set, with $\#E\gtrsim \de^{-s+\e}$ and $\#F\gtrsim \de^{-t+\e}$. Each of $E$ and $F$ lies in a ball of radius $1/1000$ and $\dist(E,F)\ge 1/2$. Then, there exists a subset $E'\subset E$ with $\#E'\ge (1-\de^\e)\# E,$ and for every point $y\in E'$, there exist disjoint families of $\de$-tubes $\T^y=\T_1^y\sqcup\dots\sqcup\T_L^y$ (where $L=3\log(1/\de)$, and some $\T_j^y$ may be empty), with the following properties:
\begin{enumerate}[label=(\roman*)]\label{cor}
    \item\label{1} The tubes in $\T^y$ form a bush centered at $y$.
    \item\label{2} Each $\T^y_j$, if non-empty, can be writen as $\T^y_j=\sqcup_i \T^y_{j,i}$, where each $\T^y_{j,i}$ is a $(\de,\si)$-set with cardinality $\gtrsim \de^{-\si+\eta}$.
    \item\label{3} $\#(T\cap F)\sim 2^j$, for $T\in\T^y_j$.
    \item\label{4} $\T^y_j$ is either empty, or $\#(F\cap \bigcup_{T\in\T^y_j}T)\ge \de^{2\e}\#F$ in which case $\#\T^y_j\ge \de^{2\e}2^{-j}\#F$; we also trivially have $\#\T^y_j\le 2^{-j}\#F$ by (iii).
    \item\label{5} $\#(F\cap \bigcup_{T\in\T^y} T)\ge (1-\de^\e)\#F$.
    
    % $\Tau_j$ is either empty of $\#(F\cap \bigcup_{T\in\Tau_j})\ge \de^{2\e}\#F$.
    % \item \label{5}We have $\#F_{bad}\le \de^\e$, where
    % $$ F_{bad}:=F\setminus \bigcup_{j=1}^L\bigcup_{T\in\Tau_j}T. $$
\end{enumerate}
\end{corollary}
\begin{proof}[Proof of Corollary \ref{cor1}]
    We will apply Theorem \ref{discradial}. Since there are many parameters, to make less confusion, we denote the parameters appeared in Theorem \ref{discradial} by $\si(\Thm),t(\Thm),a(\Thm),E(\Thm),F(\Thm)$. And we write the parameters appeared in Corollary \ref{cor1} in the usual way as $\si, s,t,E,F$. 
    
    We first talk about the idea. To apply Theorem \ref{discradial}, we let $\si(\Thm)=\si, t(\Thm)=s,a(\Thm)=t$, and   $E(\Thm)=E,F(\Thm)=F$. We can check that the conditions in Theorem \ref{discradial} are satisfied. As a result, there exists $y\in E$ such that for all $F'\subset F$ with $\#F'\ge\de^{2\e}\#F$ (it is harmless to use $2\e$ instead of $\e$), we have
    \begin{equation}\label{goody}
       \cH^\si_{\de,\infty}(\pi_y(F'))>\de^{\eta/2}.
    \end{equation}
 
We will iteratively use Theorem \ref{discradial} to obtain a lot of $y$ that satisfies \eqref{goody}, we will let $E'$ to be the set of these $y$'s and our $\T^y$ will be constructed using \eqref{goody}.

We talk about the details. Suppose we have obtained $\{y_1,\dots,y_N\}$ such that \eqref{goody} is true for each of these $y_i$. If $N<(1-\de^\e)\#E$, then we let $E(\Thm)=E\setminus\{y_1,\dots,y_N\}$. We see that $\#E(\Thm)>\de^{-s+2\e}$. If we let $\e(\Thm)=2\e$, then we can apply Theorem \ref{discradial} and obtain $y_{N+1}$ that satisfies \eqref{goody}. By iteration, we obtain $E'\subset E$ with $\#E'\ge (1-\de^\e)\#E$ such that for each $y\in E'$, we have: if $F'\subset F$ with $\#F'\ge \de^{2\e} \#F$, then
\begin{equation}\label{goody2}
    \cH^\si_{\de,\infty}(\pi_y(F'))>\de^{\eta/2}. 
\end{equation} 

Our next step is to construct $\T^y=\sqcup_{1\le j\le L} \T^y_j$ for each $y\in E'$. We fix a $y\in E'$ in the rest of proof. The idea is to iteratively use \eqref{goody2}.

We first choose a set of $\de$-caps $\cC=\{C\}\subset \S^{n-1}$ that forms a partition of $\S^{n-1}$. For each cap $C$, let $T$ be a $\de$-tube that passes through $y$ and points to direction $C$. 
In this way, $\cC$ naturally corresponds to $\T$ which is a full bush centered at $y$. The reader can check $\#\T\sim \de^{-(n-1)}$. Our $\T^y$ will be constructed as a subset of $\T$.

We first let $F'=F$, and then of course $\#F'\ge \de^{2\e}\#F$, so we have \eqref{goody2}. Let $\T'\subset \T$ be the tubes that intersect $F'$, then \eqref{goody2} is equivalent to saying that $\T'$ satisfies that \[\cH^\si_{\de,\infty}(\pi_y(\cup_{T\in\T'}T))>\de^{\eta/2}.\]
For $1\le j\le L=3\log(1/\de)$, define
\[ \T'_j=\{T\in \T': \#(T\cap F')\sim 2^j\}. \]
We obtain a partition $\T'=\sqcup_j \T'_j$.
By pigeonholing, there exists $j$ such that
\[ \cH^\si_{\de,\infty}(\pi_y(\cup_{T\in\T'_j}T))\gtrsim \de^{\eta}. \]
By Lemma \ref{usefullemma2.5}, we obtain a subset $\T^y_{j,1}\subset \T'_j$, so that $\T^y_{j,1}$ is a $(\de,\si)$-set with cardinality $\gtrsim \de^{-\si+\eta}$. Next, we let $F'=F\setminus \bigcup_{T\in\T^y_{j,1}}T$ and then check whether $\#F'\ge\de^{2\e}\#F$. If not, we stop. If yes, we repeat the argument above and obtain $\T^y_{j',1}$ or $\T^y_{j,2}$.

Suppose that we have obtained $\T^y_j$ $(j=1,\dots,L)$, where each $\T^y_j=\sqcup_{1\le i\le i(j)}\T^y_{j,i}$. Also, each $\T^y_{j,i}$ satisfies \ref{2}, and each $T\in\T^y_{j}$ satisfies \ref{3}.
We let 
\[F'=F\setminus \bigcup_{T\in \cup_j\T^y_j}T.\]
If $\#F'\ge \de^{2\e}\#F$, then we repeat the argument and obtain some $\T^y_{j,i(j)+1}$. We redefine $\T^y_j$ to be the disjoint union $\T^y_{j,i(j)+1}\sqcup \T^y_j$, and redefine $i(j)$ to be $i(j)+1$.
If $\#F'<\de^{2\e}\#F$, then we stop.

Suppose we stop. For the purpose of \ref{4}, define the significant set of $j$ to be
\[ J=\{j: \#(F\cap \bigcup_{T\in\T^y_j}T)\ge \de^{2\e}\#F \}. \]

We throw away those $\T^y_j$ for $j\notin J$, and let $\T^y=\bigsqcup_{j\in J}\T^y_j$. Finally, we check \ref{5}. We note that
\[ \#(F\cap \bigcup_{T\in\T^y}T)= \#F-\#(F\setminus \bigcup_{T\in\cup_j\T^y_j}T)-\sum_{j\notin J}\#(F\cap \cup_{T\in\T^y_j}T ). \]
This is bounded from below by $\#F-\de^{2\e}\#F-3\log(1/\de)\de^{2\e}\#F\ge (1-\de^\e)\#F$, when $\de$ is small enough.

\end{proof}

Let us return to the proof of Proposition \ref{discliu}. Since
$$ s>\max\{k+s-t,0\}, $$
we can apply Corollary \ref{cor} with $\si:=s$.
We obtain a set $E'\subset E$ with $\#E'\ge (1-\de^{4\e})\#E$, and for all $y\in E'$ the tubes $\T^y=\T_1^y\sqcup\dots\sqcup\T_L^y$ ($L=3\log(1/\de)$) satisfying the properties in Corollary \ref{cor}. $\T^y$ is a bush of tubes centered at $y$. Next, we will estimate the number of pairs $(y,x)\in E'\times F$ that satisfy certain properties. To make the expression easier, for any set of tubes $\T'$, we write $\bigcup \T':= \bigcup_{T\in\T'}T$. 

% Set $\T_j'=\cup_{y\in E'} \T_j^y$, $\T=\cup_j \T_j'=\cup_{y\in E'} \T^y$. We make the convention: If two tubes $T,T'$ lie within the $\de$-neighborhood of the same line, then we treat $T,T'$ as the same tube. There is another subtlety that $\T^y$ may not equal to $(\T)^y$ defined by \eqref{Ty}, but we always have $\T^y\subset (\T)^y$, since $\T^y\subset \T$ and any $T\in\T^y$ contains $y$.

By \ref{5}, we have
\begin{equation}\label{incidence0}
    \#\{ (y,x)\in E'\times F: x\in \bigcup\T^y \}=\sum_{y\in E'}\#\left(F\cap \bigcup\T^y\right)\ge(1-\de^{4\e})\#E'\#F.
\end{equation} 

Now, we make a counter assumption: \eqref{liubigpoint} fails for all $x\in F$. Thus for every $x\in F$, there exists a subset $E_x'\subset E$ such that $\#E_x'\ge \de^\e\#E$, and
\begin{equation}\label{cardiproj}
    |\pi_x(E'_x)|_\de<\de^{-s+\tau_0}.
\end{equation}
Since $\#E'\ge (1-\de^{4\e})\#E$, we have $\#(E'_x\cap E')\gtrsim \de^\e\#E$. We may assume $E'_x\subset E'$ by replacing $E'_x$ with $E'_x\cap E'$. For each $x\in F$, we choose a bush $\Tau_x$ centered at $x$, consisting of $\de$-tubes, so that $\Tau_x$ covers $E'_x$ and 
\begin{equation}\label{tx}
    \#\Tau_x=|\pi_x(E_x')|_\de<\de^{-s+\tau_0}.
\end{equation}
We immediately have
\begin{equation}\label{incidence1}
    \#\{ (y,x)\in E'\times F: y\in \bigcup\cT_x \}=\sum_{x\in F}\#\left(E'\cap\bigcup\Tau_x\right) \ge \de^\e \#E'\#F.
\end{equation}
The inequalities \eqref{incidence0} and \eqref{incidence1} together imply
\begin{equation}\label{incidence2}
    \#\{(y,x)\in E'\times F:x\in \bigcup\T^y,y\in \bigcup\cT_x \}\ge (\de^\e-\de^{4\e})\#E'\#F.
\end{equation}

By pigeonholing, there exists a $j$ such that
\begin{equation}\label{incidence3}
    \#\{(y,x)\in E'\times F:x\in \bigcup\T_j^y,y\in \bigcup\cT_x \}\gtrsim \de^{2\e}\#E'\#F.
\end{equation}

Next, we introduce the high-density  tubes:
\begin{equation}\label{tjyh}
     \T_j^{y,h}:=\{T\in\T_j^y: \#\{y\in E': T\in\T_j^y\}\ge \de^{-\tau_0/2}\}. 
\end{equation}
Also define the low-density tubes $\T_j^{y,l}:=\T_j^y\setminus \T_j^{y,h}$.
We want to show that
\begin{equation}\label{incidence4}
    \#\{(y,x)\in E'\times F:x\in \bigcup\T_j^{y,h},y\in \bigcup\cT_x \}\gtrsim \de^{2\e}\#E'\#F.
\end{equation}
To show this, it suffices to show
\begin{align}
    \nonumber&\#\{(y,x)\in E'\times F:x\in \bigcup\T_j^{y,l},y\in \bigcup\cT_x \}\\
    \label{lowtube}=&\sum_{x\in F}\#\left\{y\in E': x\in \bigcup\T_j^{y,l},y\in \bigcup\cT_x\right\}\lesssim \de^3\e \#E'\#F.
\end{align}
For fixed $x\in F$, we note that if $y\in\left\{y\in E': x\in \bigcup\T_j^{y,l},y\in \bigcup\cT_x\right\}$, then there exists $T\in \T_j^{y,l}\cap \cT_x$ such that $x,y\in T$. Therefore, we can bound \eqref{lowtube} by
\[ \le \sum_{x\in F}\sum_{T\in\cT_x}\#\{y\in E': T\in \T_j^{y,l}\}. \]
By the definition of $\T^{y,l}_j$, we see that if $T\in \T^{y,l}_j$ for some $y$, then $\#\{y\in E': T\in \T^y_j\}< \de^{-\tau_0/2}$. Therefore, we bound the inequality above by
\[ \lesssim \sum_{x\in F} \#\cT_x \de^{-\tau_0/2}\lesssim \#F\de^{-s+\tau_0/2}\lesssim \de^{3\e}\#E'\#F, \]
if $\e$ is small enough depending on $\tau_0$.
This prooves \eqref{lowtube} and hence \eqref{incidence4}.

Next, we show that there exists $E''\subset E'$ with $\#E''\gtrsim \de^{2\e} \#E'$, such that for $y\in E''$: 
$$\#\T_j^{y,h}\ge \de^{2\e} \#\T_j^y.$$
Note that
\begin{align*}
    \de^{2\e} \#E'\#F\lesssim \#\{(y,x)\in E'\times F:x\in \bigcup\T_j^{y,h},y\in \bigcup\cT_x \}\\
    =\sum_{y\in E'}\#\{x\in F:x\in \bigcup\T_j^{y,h},y\in \bigcup\cT_x \}. 
\end{align*} 
By pigeonholing, we can choose $E''\subset E'$ with $\#E''\gtrsim \de^{2\e}\#E'$ so that for $y\in E''$, 
\[\#\{x\in F:x\in \bigcup\T_j^{y,h},y\in \bigcup\cT_x \}\gtrsim \de^{2\e} \#F.\]
Since $\T_j^{y,h}\subset \T_j^y$
and each $T\in\T_j^y$ satisfies $\#(F\cap T)\sim 2^j$, 
we have 
\[\#\{x\in F:x\in \bigcup\T_j^{y,h},y\in \bigcup\cT_x \}\le \#(F\cap \bigcup\T_j^{y,h}) \sim 2^j \#\T_j^{y,h},\]
which implies for $y\in E''$,
\begin{equation}\label{tyhbigty}
    \#\T_j^{y,h}\gtrsim \de^{2\e} 2^{-j}\#F\gtrsim \de^{2\e} \#\T_j^y.
\end{equation}

We define $\T_j^{h}=\cup_{y\in E''}\T^{y,h}_j$. Since we will consider the interplay among $\{\T^{y,h}_j\}_{y\in E''}$, we make the convention: If two tubes $T,T'$ lie within the $\de$-neighborhood of the same line, then we treat $T,T'$ as the same tube. One good example to keep in mind is that: Suppose $x,x'\in F, y\in E'$ with $x,x',y$ lying on the same line. If $T$ contains $x,y$ and $T'$ contains $x',y$, then we identify $T,T'$ as the same tube. As a result, the tubes from different $\T^{y,h}_j$ that are identified as the same tube will appear once in $\T^h_j$. This potentially could result in that $\#\T^y_j$ can be much smaller than $\sum_{y\in E''} \#\T^{y,h}_j$. However, we claim the following estimate
\begin{equation}\label{cl}
    \#\T^{h}_j \gtrsim \de^{O(\eta+\e)}\sum_{y\in E''}\#\T^{y,h}_j. 
\end{equation}
We prove the claim.
Recall that $\T^y_j=\sqcup_{i=1}^{i(y)} \T^y_{j,i}$ where each $\T^y_{j,i}$ is a $(\de,s)$-set with cardinality $\gtrsim \de^{-s+\eta}$ (see \ref{2}). Here $i(y)$ is the number that indicates the cardinality of $\T^y_j$:
\begin{equation}\label{cardofT}
    i(y)\de^{-s+\eta}\lesssim \T^y_j\lesssim i(y)\de^{-s}.
\end{equation}

Note that $E''$ is a $(\de,s)$-set with $\#E''\gtrsim \de^{O(\e)-s}$. By pigeonholing, we can find $E^\circ\subset E''$ with $\#E^\circ\gtrsim \de^{O(\e)-s}$, such that $\#\T_j^{y,h}$ are comparable for all $y\in E^\circ$ and
\[ \sum_{y\in E''}\#\T^{y,h}_j\lesssim \de^{-\e}\sum_{y\in E^\circ}\#\T^{y,h}_j. \]
Now that $E^{\circ}$ is a $(\de,s)$-set with $\#E^{\circ}\gtrsim \de^{O(\e)-s}$, by Lemma \ref{usefullemma3},
we can choose a $(\de\de^{-C(\e+\eta)},s)$-set $E^*\subset E^\circ$ with $\#E^*\gtrsim \de^{O(\e+\eta)-s}$ ($C$ is some large number to be determined later). We have
\begin{equation}\label{phargument}
    \sum_{y\in E''}\#\T^{y,h}_j\lesssim \de^{-O(\e+\eta)}\sum_{y\in E^*}\#\T^{y,h}_j.
\end{equation}

We are ready to estimate the lower bound of $\#T_j^h$.
We have
\begin{align}
    \label{furstenberg}\# \T_j^{h} \ge \# \left(\bigcup_{y\in E^*} \T_j^{y,h}\right) &\geq \#\left(\bigcup_{y\in E^*} \left(\T_j^{y,h}\setminus \bigcup_{y'\in E^*\setminus \{y\}} \T_j^{y',h}\right)\right) \\
    (\textup{by~the~disjointness})&= \sum_{y\in E^*} \#\left(\T_j^{y,h} \setminus \bigcup_{y'\in E^*\setminus\{y\}} \T_j^{y',h}\right) \\
    &\ge \sum_{y\in E^*} \#\left(\T_j^{y,h} \setminus \bigcup_{y'\in E^*\setminus\{y\}} \T_j^{y'}\right) \\
    &\ge \sum_{y\in E^*} \bigg(\# \T_j^{y,h} - \sum_{y'\in E^*\setminus\{y\}}\#\left(\T_j^{y'} \cap  \T_j^{y}\right)\bigg).
\end{align}
We show that
\begin{equation}\label{biggeronehalf}
    \# \T_j^{y,h} - \sum_{y'\in E^*\setminus\{y\}}\#\left(\T_j^{y'} \cap  \T_j^{y}\right)\ge \frac{1}{2}\#\T_j^{y,h}. 
\end{equation} 
For fixed $y,y'$, we want to find an upper bound for 
\[\#\left(\T_j^{y} \cap  \T_j^{y'}\right).\] 
This is less than
\[ \#\{ T\in \T_j^{y}: y'\in T \}.  \]  Since $\T^{y}_j=\sqcup_{i=1}^{i(y)} \T^{y}_{j,i}$ where each $\T^{y}_{j,i}$ is a $(\de,s)$-set,       we have 
\begin{align*}
    \#\{ T\in \T_j^{y}: y'\in T \}\le \sum_{i=1}^{i(y)}\#\{ T\in \T_{j,i}^{y}: y'\in T \}\lesssim i(y) |y-y'|^{-s}.
\end{align*}
So, we have
\begin{align*}
    \sum_{y'\in E^*\setminus\{y\}}\#(\T_j^{y}\cap\T_j^{y'})&\lesssim i(y)\sum_{y'\in E^*\setminus\{y\}} |y-y'|^{-s}\\
    &=i(y)\sum_{\de\de^{-C(\e+\eta)}\le d\le 1}\ \ \sum_{y'\in E^*, |y-y'|\sim d} d^{-s}.
\end{align*} 
Here the summation over $d$ is over dyadic numbers.
Since $E^*$ is a $(\de\de^{-C(\e+\eta)},s)$-set, we have $\#(E^*\cap B_d(y))\lesssim (\frac{d}{\de\de^{-C(\e+\eta)}})^{s}$, the expression above is bounded by
\begin{align}\label{abo}
    &\lesssim i(y)\sum_{\de\de^{-O(\e+\eta)}\le d\le 1} \de^{sC(\e+\eta)}(\frac{d}{\de})^s d^{-s}\lesssim i(y)\de^{-s}\de^{sC(\e+\eta)}|\log\de|.
\end{align}
On the other hand, by \eqref{tyhbigty} and \eqref{cardofT}, we get $\#\T^{y,h}_j\gtrsim \de^{2\e}\#\T^{y}_j\gtrsim i(y)\de^{2\e+\eta-s}$. Therefore, if $C$ is large enough, then the right hand side of \eqref{abo} is $\le \frac{1}{2}\#\T^{y,h}_j$. So, we proved \eqref{biggeronehalf}.
If we look back to \eqref{furstenberg}, we obtain 
\[ \#\T^h_j\ge \frac{1}{2}\sum_{y\in E^*}\#\T^{y,h}_j. \]
Combining with \eqref{phargument}, we proved the claim \eqref{cl}.

\medskip

Estimating the right hand side of \eqref{cl} using \eqref{tyhbigty} and \ref{4}, we obtain
\begin{equation}\label{lowerbound}
    \#\T_j^{h}\gtrsim \de^{O(\eta+\e)}\de^{-s}2^{-j}\#F.
\end{equation}

Finally, we estimate $I(E',\T_j^{h}):=\{(y,T)\in E'\times \T_j^h: T\in \T_j^y\}$.
Recalling $\T^h_j=\cup_{y\in E''} \T_j^{y,h}$ and the definition of $\T_j^{y,h}$ in \eqref{tjyh}, we have the lower bound
\begin{equation}\label{lowerbound2}
    I(E',\T_j^{h})= \sum_{T\in \T_j^h}\#\{y\in E': T\in \T^y_j\}\ge \#\T_j^{h} \de^{-\tau_0/2}.
\end{equation} 
We have the upper bound
\begin{align*}
    I(E',\T^{h}_j)&\le (\#\T_j^{h})^{1/2}\left( \sum_{T\in \T^{h}_j}\#\{y\in E':T\in \T_j^y\}^2 \right)^{1/2}\\
    &=(\#\T_j^{h})^{1/2}\left( \sum_{y, y'\in E'}\#\{T\in \T_j^h: T\in \T^y_j\cap \T^{y'}_j\} \right)^{1/2}\\
    &=(\#\T_j^{h})^{1/2}\left( \sum_{y\neq y'\in E'}\#\{T\in \T_j^h: T\in \T^y_j\cap \T^{y'}_j\}+I(E',\T_j^{h}) \right)^{1/2}
\end{align*}
Note that $\#\{T\in \T_j^h: T\in \T^y_j\cap \T^{y'}_j\}\le\#\{T\in \T_j^y:y'\in T\}$,
and by \ref{2} (with $\si=s$), it further has bound $\lesssim \de^{-O(\eta)} (\de/|y-y'|)^s\#\T_j^y.$ Here, we used the fact that the tubes in $\{T\in \T^y_j: y'\in T\}$ are contained in a $(\de/|y-y'|)$-tube. So, 
\begin{align*}
    \#\{T\in \T^y_j: y'\in T\}=\sum_i \#\{T\in \T_{j,i}^y:y'\in T\}\lesssim \sum_i (1/|y-y'|)^s\\
    \lesssim \sum_i (1/|y-y'|)^s \de^{-\eta+s}\#\T_{j,i}^y=\de^{-\eta}(\de/|y-y'|)^s\#\T_j^y. 
\end{align*} 

We see that
\begin{align*}
    \sum_{y\neq y'\in E'}\#\{T\in \T_j^y:T\in \T^y_j\cap \T^{y'}_j\}
    \lesssim \de^{-O(\eta)}\sum_{y\neq y'\in E'}\left(\frac{\de}{|y-y'|}\right)^s\#\T_j^y.
\end{align*}
Noting that from \ref{4} that $\#\T_j^y\le 2^{-j}\#F$ and 
\[ \sum_{y\neq y'\in E'}\left(\frac{\de}{|y-y'|}\right)^s=\sum_{y\in E'}\sum_{y'\in E'\setminus \{y\}} \left(\frac{\de}{|y-y'|}\right)^s\lesssim \log \de^{-1}\#E'\lesssim \de^{-\e}\de^{-s}, \]
we obtain
\[ I(E',\T^h_j)\lesssim (\#\T^h_j)^{1/2}\left( \de^{-O(\e+\eta)} \de^{-s}2^{-j}\#F+I(E',\T^h_j)\right)^{1/2}. \]

This implies:
$$ I(E',\T^{h}_j)\lesssim \de^{-O(\eta+\e)}\bigg((\#\T_j^{h})^{1/2}(\de^{-s}2^{-j}\#F)^{1/2}+\#\T_j^{h}\bigg). $$
Comparing with the lower bound \eqref{lowerbound2}, when we choose $\eta,\e$ sufficiently small compared with $\tau_0$, we obtain
$$ \#\T_j^{h}\lesssim \de^{-O(\eta+\e)+\tau_0/2}\de^{-s}2^{-j}\#F,$$
which contradicts \eqref{lowerbound}, since $\eta,\e$ can be chosen much smaller than $\tau_0$.

\bibliographystyle{abbrv}
\bibliography{bibli}
%***************************************

\end{document}